\documentclass[12pt,a4paper,leqno]{amsart}
\usepackage{etoolbox}
\ifdef{\warningsaserrors}{%
  }{}

\usepackage{amsmath}
\usepackage{amsthm}
\usepackage{amssymb}
\usepackage{amsfonts}
\usepackage{graphicx}
\usepackage[usenames,dvipsnames,svgnames,table]{xcolor}
\usepackage{esint}
\usepackage{enumerate}
\usepackage{fullpage}
\usepackage{verbatim}
\usepackage{booktabs}
\usepackage{hyperref}
\usepackage{cleveref}
\usepackage{todonotes}

\newtheorem{theorem}{Theorem}

\newtheorem{lemma}[theorem]{Lemma}
\newtheorem{corollary}[theorem]{Corollary}

\theoremstyle{remark}
\newtheorem{remark}[theorem]{Remark}

\newcommand{\R}{\mathbf{R}}

\newcommand{\N}{\mathbf{N}}
\newcommand{\Z}{\mathbf{Z}}
\newcommand{\eps}{\varepsilon}
\newcommand{\cO}{\mathcal{O}}
\newcommand{\cC}{\mathcal{C}}
\newcommand{\cK}{\mathcal{K}}
\newcommand{\cF}{\mathcal{F}}

\newcommand{\cT}{\mathcal{T}}
\newcommand{\Energy}{\mathcal{E}}

\newcommand{\p}{\partial}

\newcommand{\dd}[1]{\;d#1}
\renewcommand{\:}{\; : \;}

\DeclareMathOperator{\sgn}{sgn}

\title{Supercritical equivariant biharmonic maps from $\R^5$ into $S^5$}
\author{M.K.\ Cooper}

\begin{document}
\begin{abstract}
  We study supercritical $O(d)$-equivariant biharmonic maps with a focus on $d = 5$, where $d$ is the dimension of the domain.
  We give a characterisation of non-trivial equivariant biharmonic maps from $\R^5$ into $S^5$ as heteroclinic orbits of an associated dynamical system.
  Moreover, we prove the existence of such non-trivial equivariant biharmonic maps.
  Finally, in stark contrast to the harmonic map analogue, we show the existence of an equivariant biharmonic map from $B^5(0, 1)$ into $S^5$ that winds around $S^5$ infinitely many times.
\end{abstract}

\maketitle

\section{Introduction}
Our main purpose is to extend the analysis of the author in~\cite{MKC15}, which studies equivariant (extrinsic) biharmonic maps in the (energy)-critical regime, to the (energy)-supercritical regime.
Due to technical obstacles, which we discuss below, we are only able to extend our analysis to the cases $d \in \{5, 6, 7\}$, with a particular emphasis on the $d = 5$ case, where $d$ is the dimension of our domain.

Next, we will introduce (extrinsic) biharmonic maps from flat domains into spheres.
Of course, one can consider biharmonic maps from more general domains into more general targets, but we do not need that generality.
Let $d, n \in \N$ and $\Omega \subset \R^d$ be a bounded domain.
We consider the target $S^n$ to be isometrically embedded in $\R^{n + 1}$ as
\begin{equation*}
  S^n := \{ x \in \R^{n + 1} \: |x| = 1 \}.
\end{equation*}
Consider the bi-energy:
\begin{equation*}
  E_2[u] := \int_{\Omega} |\Delta u|^2 \dd{x} \text{ for } u \in H^2_g(\Omega; S^n),
\end{equation*}
where $g$ is boundary data,
\begin{equation*}
  H^2(\Omega; S^n) := \{ u \in H^2(\Omega; \R^{n+1}) \: |u(x)| = 1 \text{ for a.e.\ } x \in \Omega \},
\end{equation*}
and $H^2_g(\Omega; S^n) \subset H^2(\Omega; S^n)$ such that $u \in H^2_g(\Omega; S^n)$ when $D^\alpha u = D^\alpha g$ on $\p \Omega$ for $|\alpha| \leq 1$.

The biharmonic maps under consideration here are critical points of $E_2$.
The Euler-Lagrange equation of $E_2$ is
\begin{equation}\label{eq:BiEulerLagrange}
  \left\{
  \begin{aligned}
    \Delta^2 u(x) &\perp \cT_{u(x)} S^n \text{ in } \Omega \text{ and}
    \\
    D^\alpha u &= D^\alpha g \text{ on } \p \Omega \text{ for } |\alpha| \leq 1,
  \end{aligned}
  \right.
\end{equation}
where we interpret the above in a distributional sense for $u \in H^2(\Omega; S^n)$.

One can view biharmonic maps as a higher-order analogue of harmonic maps which are critical points of the first-order Dirichlet energy
\begin{equation*}
  E_1[u] := \int_\Omega |Du|^2 \dd{x}.
\end{equation*}
We use the equivariant ansatz
\begin{equation}\label{eq:EquiAnsatzEq}
  u(x)
  =
  \begin{cases}
    \left( \frac{x}{|x|} \sin\psi(|x|), \cos\psi(|x|) \right) &\text{ for } x \in \overline{B^d(0, 1)} \setminus \{0\},
    \\
    \hat{e}_{d+1} &\text{ if } x = 0
  \end{cases}
  =:
  \Upsilon(\psi)(x).
\end{equation}
To ensure that $u = \Upsilon(\psi)$ is continuous at the origin, we set $\psi(0) = 0$.
This isn't the only value of $\psi(0)$ that achieves this, but due to the symmetry of the situation we may assume, without loss of generality, that $\psi(0) = 0$.
For a detailed discussion of this ansatz and its equivalence to $O(d)$-equivariance, see~\cite[Section 1 and Section 2]{MKC15}.

Assuming $\psi \in C([0, 1]; \R) \cap C^4((0, 1]; \R)$, and substituting $u = \Upsilon(\psi)$ into~\eqref{eq:BiEulerLagrange} gives the following ODE for $\psi$:
\begin{equation}\label{eq:PsiProblem}
  \left\{
    \begin{aligned}
      \psi^{(4)}
      &=
      6 (\psi')^2 \psi''
      + \frac{2 (d-1)}{r} \left((\psi')^3 - \psi^{(3)}\right)
      - \frac{d-1}{r^2} \big((d - \cos(2\psi) - 4) \psi''
      \\
      &\quad
      + \sin(2 \psi) (\psi')^2\big)
      + \frac{(d-3) (d-1)}{r^3} (\cos(2 \psi) + 2) \psi'
      \\
      &\quad
      - \frac{3 (d-3) (d-1)}{2 r^4} \sin(2 \psi) \text{ for } r \in (0, 1) \text{ and}
      \\
      \psi(0) &= 0,
    \end{aligned}
  \right.
\end{equation}
We carried out this calculation by making the obvious modifications to the Mathematica code presented in~\cite[p.\ 2902]{MKC15}.
Note that the boundary conditions on $u$ in~\eqref{eq:BiEulerLagrange} turn into conditions on $\psi(1)$ and $\psi'(1)$.
Although the above definitions are only presented for the domain $B^d(0, 1)$, we can extend these definitions to balls of arbitrary radius centred at the origin or $\R^d$ in the obvious way.

This ODE for $\psi$ is dilation invariant.
We make the change of variables $\psi(r) = \phi(s)$, where $s = \log r$, to arrive at an autonomous ODE for $\phi$:
\begin{equation}\label{eq:FullODE}
  \begin{aligned}
    \phi^{(4)} &= ((d-1) \cos(2 \phi)-(d-11) d-21) \phi'' - \frac{3}{2} (d-3) (d-1) \sin(2 \phi)
    \\
    &\quad
    + \left(6 \phi''-(d-1) \sin(2 \phi)\right) (\phi')^2
    + (d-4) ((d-1) \cos (2 \phi)+3 d-5) \phi'
    \\
    &\quad
    + 2 (d-4) (\phi')^3
    -2 (d-4) \phi^{(3)}.
  \end{aligned}
\end{equation}
The $\psi(0) = 0$ condition translates to the condition $\phi(s) \rightarrow 0$ as $s \rightarrow -\infty$.
We will work with~\eqref{eq:FullODE}, mostly forgetting about the ODE in~\eqref{eq:PsiProblem}.

Now let us focus on the coefficient of $\phi''$ in the first term on the RHS of~\eqref{eq:FullODE}, namely
\begin{equation*}
  (d-1) \cos(2 \phi)-(d-11) d-21.
\end{equation*}
The qualitative properties of~\eqref{eq:FullODE} largely depend on the sign of this coefficient.
We make the following observations:
\begin{itemize}
\item
  for $3 \leq d \leq 7$ this coefficient is strictly positive;
\item
  for $d \leq 2$ or $d \geq 10$ this coefficient is strictly negative; and
\item
  for $d = 8$ or $d = 9$ the coefficient changes sign.
\end{itemize}
From this, we would suppose that different techniques would be necessary for the analysis in these different regimes.
As we will remark below, this is in stark contrast to the harmonic map analogue of our problem.

Next, we would like to explain why we focus our attention on the $d = 5$ case.
Clearly, many of the terms in~\eqref{eq:FullODE} vanish when $d = 4$.
The introduction of these terms when transitioning to the $d = 5$ case cause many difficulties.
We don't believe that the qualitative nature of these difficulties change when going from $d = 5$ to $d \in \{6, 7\}$.
However, the degree of these difficulties seem to increase, and some of our techniques fail.
We will try to remark when these failures occur.
This is the reason we focus on the $d = 5$ case.

Some of our results hold for a larger range of $d$ other than just $d = 5$.
However, we only present these more general results when it does not obscure the ideas and arguments in the $d = 5$ case.
For reference, if $d = 5$ then~\eqref{eq:FullODE} becomes
\begin{equation}\label{eq:5dODE}
  \begin{aligned}
    \phi^{(4)} &= (4 \cos(2 \phi) + 9) \phi'' - 12 \sin(2 \phi)
    + \left(6 \phi''- 4 \sin(2 \phi)\right) (\phi')^2
    + (4 \cos (2 \phi) + 10) \phi'
    \\
    &\quad
    + 2 (\phi')^3
    -2 \phi^{(3)}.
  \end{aligned}
\end{equation}
One of the most interesting things regarding this work is the contrast it shows with the harmonic map analogue.
Therefore, before we state our main results we give a high-level description of this analogue.
The harmonic map analogue of~\eqref{eq:FullODE} is
\begin{equation}\label{eq:HarmonicPhi}
  \phi''_H = \frac{d - 1}{2} \sin(2 \phi_H) - (d - 2) \phi'_H.
\end{equation}
Note that in the harmonic map case $d = 2$ is the critical dimension, and the equation is supercritical for $d \geq 3$.
For the same reasons as in the setup for the biharmonic map case, we have $\phi_H(s) \rightarrow 0$ as $s \rightarrow -\infty$.
In the supercritical regime, there is only one orbit, up to the symmetries of the problem, that satisfies~\eqref{eq:HarmonicPhi}, and this is the heteroclinic orbit connecting $(\phi_H, \phi_H') = 0$ to $(\phi_H, \phi_H') = (\pi/2, 0)$.

Next we explain why a supercritical harmonic map must satisfy $|\phi_H| < \pi$, that is, such a harmonic map can't wrap around past the south-pole of its target.
One can view~\eqref{eq:HarmonicPhi} as an equation modelling a pendulum with friction.
The energy of this pendulum is
\begin{equation*}
  \Energy_H[\phi_H] := \frac{1}{2} (\phi_H')^2 + \frac{d - 1}{2} \cos^2\phi_H.
\end{equation*}
The potential energy is $\frac{d - 1}{2} \cos^2\phi_H$.
The friction means that this energy is monotone decreasing:
\begin{equation*}
  \p_s \Energy_H[\phi_H] = -(d - 2) (\phi_H')^2.
\end{equation*}
From this we see that the potential energy of $\phi_H$ must be below the maximum possible potential energy which occurs at integer multiples of $\pi$.
Another interesting observation is that in the supercritical regime for~\eqref{eq:HarmonicPhi} the qualitative properties of the solutions does not change with changing $d$ like we saw for the biharmonic case.

We now move onto stating our main results.
We show that if $d \in \{5, 6, 7\}$ and $\psi$ solves~\eqref{eq:PsiProblem}, then $u = \Upsilon(\psi)$ is a smooth biharmonic map, that is, any potential problem with~\eqref{eq:PsiProblem} at the origin does not occur.

\begin{theorem}\label{thm:SmoothBiharmonicMap}
  Let $d \in \{5, 6, 7\}$, $\psi \in C([0,1];\R) \cap C^{\infty}((0,1];\R)$, with $\psi(0) = 0$, be a solution to~\eqref{eq:PsiProblem}, and
  \begin{equation*}
    u = \Upsilon(\psi) \in C(\overline{B^d(0,1)}; S^d) \cap C^{\infty}(\overline{B^d(0,1)} \setminus \{0\}; S^d).
  \end{equation*}
  Then $u \in C^\infty(\overline{B^d(0,1)}; S^d)$ and it is biharmonic.
\end{theorem}

The following theorem characterises equivariant biharmonic maps from $\R^5$ into $S^5$.

\begin{theorem}\label{thm:EntireSolutionsAreHeteroclinic}
  Let $d = 5$ and $\psi \in C([0,\infty); \R) \cap C^{\infty}((0,\infty); \R)$ be a non-trivial solution to~\eqref{eq:PsiProblem}, with $\psi(0) = 0$.
  When we consider $\phi(s) = \psi(e^s)$ which solves~\eqref{eq:5dODE}, $\phi$ is a heteroclinic orbit connecting the origin and either $(-\pi/2, 0, 0, 0)$ or $(\pi/2, 0, 0, 0)$.
  Moreover, such a $\psi$ exists.
\end{theorem}

With~\cite[Theorem 3]{MKC15} the author proves an analogue of this result in the critical, that is, $d = 4$ case.
However, the result in~\cite{MKC15} is stronger, because it proves that the heteroclinic orbit is unique, up to symmetries of the problem.
The proof of this uniqueness rests on two main facts about~\eqref{eq:FullODE} when $d = 4$.
Firstly, we have an explicit expression for a heteroclinic orbit.
Secondly, this explicit expression makes some key terms in~\eqref{eq:FullODE} vanish.
From this a comparison principle between this explicit orbit and other orbits in $W^u(0)$ follow, where $W^u(0)$ is the unstable manifold of the origin of the associated first-order system in $(\phi, \phi', \phi'', \phi^{(3)})$.
A consequence of the comparison principle is that this explicit orbit repels every other orbit in $W^u(0)$ such that these other orbits eventually blow up in finite $s$-time.

Unfortunately, in the $d = 5$ case we are unable prove the uniqueness, up to symmetry, of the heteroclinic orbit.
The two mains facts that our proof in the $d = 4$ case rest upon are not true in the $d = 5$ case.
However, numerical studies give a strong suggestion that indeed the heteroclinic orbit is unique up to the symmetries of our problem.

Finally, we show that, in stark contrast to the harmonic case, there are equivariant biharmonic maps from $B^5(0,1)$ into $S^5$ that wind around $S^5$ infinitely many times.

\begin{theorem}\label{thm:InfiniteWrapAround}
  There exists a biharmonic map
  \begin{equation*}
    u = \Upsilon(\psi) \in C^{\infty}(B^5(0, 1); S^5),
  \end{equation*}
  such that $\psi(0) = 0$ and $\psi(r) \rightarrow \infty$ as $r \nearrow 1$.
\end{theorem}

We believe that another strength of this work is its contribution to the theory of fourth-order ODE\@.
It extends the geometric approach of Hofer and Toland in~\cite{HoferTolandHamiltonian} to an equation having features far beyond what current general theory covers.

In addition to the work in~\cite{HoferTolandHamiltonian}, our approach is deeply inspired by the work of van den Berg in~\cite{vdBergFisherKolmogorov}.
More explicitly, we discover positive invariant cones on which we can write~\eqref{eq:5dODE} as a system of two second-order ordinary differential inequalities whose dynamics are simpler to study.
The crux in showing these simplified dynamics lies in showing the non-negativity of certain functions.
Parts of these proofs, and only these proofs, are computer-assisted.
More precisely, we use interval analysis, see, for example,~\cite{AlefeldHerzbergerInterval},~\cite{KulischMirankerInterval},~\cite{Moore66Interval}, or~\cite{Moore79Interval}.
Although similar ideas appeared in the author's earlier work~\cite{MKC15}, a reading of~\cite{vdBergFisherKolmogorov} really sharpened the use of such ideas in this work.

We will now outline the structure of the rest of this paper.
In \Cref{sect:Prelims}, we go over some preliminaries, including, a description of a monotone quantity for~\eqref{eq:FullODE}.
We also describe the critical points, and linearisations around these, of~\eqref{eq:FullODE}.
In \Cref{sect:BiharmonicInUnstable}, we show that when $d \in \{5, 6, 7\}$ and $\phi$ is a solution to~\eqref{eq:FullODE} such that $\phi(s) \rightarrow 0$ as $s \rightarrow -\infty$, then $\phi$ is in $W^u(0)$.
In \Cref{sect:EventualBlowup}, we show for $d \in \{5, 6, 7\}$ and solutions $\phi$ to~\eqref{eq:FullODE} that are in $W^u(0)$ that if $|\phi''|$ becomes sufficiently large then $\phi$ blows up in finite $s$-time, that is, there exists a $s_f \in \R$ such that $|(\phi(s), \phi'(s), \phi''(s), \phi^{(3)}(s))| \rightarrow \infty$ as $s \nearrow s_f$.
In \Cref{sect:ExitingC}, we introduce a set which plays a similar role for us as the role played in~\cite{HoferTolandHamiltonian} by the connected components of the set of $(\phi, \phi'')$ that have positive potential energy.
We show that if an orbit in $W^u(0)$ of~\eqref{eq:5dODE} exits this set then it must blowup in finite $s$-time.
In \Cref{sect:HeteroOrbitsExist}, we complete the proof of \Cref{thm:EntireSolutionsAreHeteroclinic}.
In \Cref{sect:GivesSmoothBiharmonicMaps}, we prove \Cref{thm:SmoothBiharmonicMap} and \Cref{thm:InfiniteWrapAround}.

{\bf Notation.}
Throughout this paper $C$ denotes a positive universal constant.
Two different occurrences of $C$ are liable to be different.
If our constant depends on some parameter, say $\eps$, then we may indicate this by writing $C(\eps)$.

For $\Omega \subseteq \R^d$, we define $-\Omega := \{ x \in \R^d \: -x \in \Omega \}$.

{\bf Acknowledgements.}
The author is grateful to Prof.\ Dr.\ Andreas Gastel for many stimulating discussions during the initial stages of this project.

\section{Preliminaries}\label{sect:Prelims}
Later in this work we will want to refer to certain sub-expressions of~\eqref{eq:FullODE}.
Therefore, we rewrite~\eqref{eq:FullODE} as
\begin{equation}\label{eq:FullAbstractODE}
  \phi^{(4)} = q(\phi) \phi'' - f(\phi)
  + \left(6 \phi'' + \frac{1}{2} q'(\phi)\right) (\phi')^2
  +
  2(d - 4)
  \left(
  g(\phi) \phi'
  + (\phi')^3
  - \phi^{(3)}
  \right),
\end{equation}
where
\begin{equation}\label{eq:FullAbstractODESpecificSubstitutions}
  \left\{
    \begin{aligned}
      q(\phi) &= (d-1) \cos(2 \phi) - (d-11) d - 21,
      \\
      f(\phi) &= \frac{3}{2} (d-3) (d-1) \sin(2 \phi), \text{ and}
      \\
      g(\phi) &= \frac{1}{2} ((d-1) \cos (2 \phi)+3 d-5).
    \end{aligned}
  \right.
\end{equation}

Next, we derive a monotone quantity for solutions of~\eqref{eq:FullODE}.
This quantity plays a central role in our analysis.
We rearrange~\eqref{eq:FullAbstractODE}:
\begin{equation}\label{eq:EnergyDerivation1}
  \begin{aligned}
    &
    \phi^{(4)}
    - q(\phi) \phi'' + f(\phi)
    - 6 (\phi')^2 \phi''
    - \frac{1}{2} (\phi')^2 q'(\phi)
    + 2(d - 4) \phi^{(3)}
    \\
    &\quad
    =
    2(d - 4) g(\phi) \phi'
    + 2(d - 4) (\phi')^3.
  \end{aligned}
\end{equation}
We multiply this through by $\phi'$, and integrate the left hand side:
\begin{equation}\label{eq:EnergyComputationIntegral}
  \begin{aligned}
    \int \phi' \, \text{LHS} \dd{s}
    &=
    \phi' \phi^{(3)} - \frac{1}{2} (\phi'')^2 + F(\phi) - \frac{3}{2} (\phi')^4
    + 2(d - 4) \phi' \phi''
    \\
    &\quad
    - \int 2(d - 4) \, (\phi'')^2 \dd{s}
    - \frac{1}{2} q(\phi) (\phi')^2,
  \end{aligned}
\end{equation}
where $F' = f$.

We choose $F$ so that $F(0) = 0$, and hence
\begin{equation}\label{eq:PotentialBiharmonicODE}
  F(\phi) = \frac{3}{2} (d-3) (d-1) \sin^2\phi.
\end{equation}
We set
\begin{equation}\label{eq:EnergyDefn}
  \Energy[\phi]
  :=
  \phi' \left(\phi^{(3)} + 2(d - 4) \phi'' - \frac{1}{2} q(\phi) \phi' - \frac{3}{2} (\phi')^3\right)
  +
  F(\phi) - \frac{1}{2} (\phi'')^2.
\end{equation}
We call this \emph{the energy} of $\phi$, and we decompose this into the sum of $T$, the kinetic energy, and $U$, the potential energy:
\begin{equation}\label{eq:EnergySplit}
  \left\{
    \begin{aligned}
      T[\phi] &:= \phi' \left(\phi^{(3)} + 2(d - 4) \phi'' - \frac{1}{2} q(\phi) \phi' - \frac{3}{2} (\phi')^3\right) \text{ and}
      \\
      U[\phi] &:= F(\phi) - \frac{1}{2} (\phi'')^2.
    \end{aligned}
  \right.
\end{equation}
Finishing the computation, we combine the definition of $\Energy$ with~\eqref{eq:EnergyDerivation1} and~\eqref{eq:EnergyComputationIntegral}:
\begin{equation*}
  \p_s \Energy[\phi] = 2(d - 4) \left( (\phi'')^2 + g(\phi) (\phi')^2 + (\phi')^4 \right).
\end{equation*}
More explicitly, we have
\begin{equation}\label{eq:EnergyRateOfChange}
  \p_s \Energy[\phi] =
  (d-4) \left(2 (\phi'')^2 + ((d - 1) \cos(2 \phi) + 3d - 5) (\phi')^2 + 2 (\phi')^4\right).
\end{equation}
Hence for $d = 4$ the energy is a conserved quantity, and for $d > 4$ it is monotone non-decreasing.

There are some basic symmetries of~\eqref{eq:FullODE} and~\eqref{eq:EnergyDefn} that we exploit.
If $\phi$ solves~\eqref{eq:FullODE}, then for any $k \in \Z$ both $s \mapsto \phi(s) + \pi k$ and $s \mapsto \pi k - \phi(s)$ are also solutions.
Furthermore $\Energy[\phi] = \Energy[\phi + k\pi] = \Energy[k\pi - \phi]$.

Next, we consider the critical points, and linearisations around these, of~\eqref{eq:FullODE}.
We substitute $\phi^{(i)} = 0$, for $i \in \{1, 2, 3, 4\}$, into~\eqref{eq:FullODE}:
\begin{equation*}
  0 = (d-3) (d-1) \sin(2\phi(s)).
\end{equation*}
Ignoring the cases $d = 1, 3$, we see that critical points occur when $\phi = k\pi/2$, for $k \in \Z$.

The linearizations depend on the parity of $k$.
For $k$ \emph{even} the matrix associated to the linearization of~\eqref{eq:FullODE} written as a first-order system for $(\phi, \phi', \phi'', \phi^{(3)})$ is
\begin{equation*}
  \left(
    \begin{array}{cccc}
      0 & 1 & 0 & 0 \\
      0 & 0 & 1 & 0 \\
      0 & 0 & 0 & 1 \\
      -3 (d-3) (d-1) & 2 (d-4) (2 d-3) & -(d-12) d-22 & 8-2 d \\
    \end{array}
  \right)
\end{equation*}
This has the eigenvalues $3, 1, 1 - d, 3 - d$.
A corresponding eigenvector is $(1, \lambda, \lambda^2, \lambda^3)^T$, where $\lambda$ is the eigenvalue.

For $k$ \emph{odd} the matrix associated to the linearization of~\eqref{eq:FullODE} written as a first-order system for $(\phi, \phi', \phi'', \phi^{(3)})$ is
\begin{equation*}
  \left(
    \begin{array}{cccc}
      0 & 1 & 0 & 0 \\
      0 & 0 & 1 & 0 \\
      0 & 0 & 0 & 1 \\
      3 (d-3) (d-1) & 2 (d-4) (d-2) & -(d-10) d-20 & 8-2 d \\
    \end{array}
  \right)
\end{equation*}
The eigensystem in this case is not as simple as when $k$ is even.
Since we do not need to know more about this case for our analysis, we will not go into more detail.

\section{Orbits are in \texorpdfstring{$W^u(0)$}{Wu(0)}}\label{sect:BiharmonicInUnstable}
In this section we show that we may restrict our attention to the unstable manifold of the origin of~\eqref{eq:FullODE} for $d \in \{5, 6, 7\}$.

\begin{theorem}\label{thm:BiharmonicInUnstable}
  Let $s_0 \in \R$, $\phi \in C^\infty((-\infty, s_0); \R)$, and $d \in \{5, 6, 7\}$.
  Suppose $\phi$ solves~\eqref{eq:FullODE}, and $\phi(s) \rightarrow 0$ as $s \rightarrow -\infty$.
  Then $\phi'(s)$, $\phi''(s)$, and $\phi^{(3)}(s)$ also converge to zero as $s \rightarrow -\infty$.
\end{theorem}

In preparation for our proof of \Cref{thm:BiharmonicInUnstable} we prove two abstract lemmas.
The point of these lemmas is not to be the sharpest nor as general as possible, but to clarify the conceptual reasons for the truth of \Cref{thm:BiharmonicInUnstable}.

\begin{lemma}\label{lem:UnstableSetup}
  Suppose that $l, \beta, \eps_0 > 0$, $\alpha \geq \mu > 0$, and $\xi \in C^3([0, \infty); \R)$ such that $\xi'(0) \leq -l$, $\xi''(0) = 0$, $\xi^{(3)}(0) \geq 0$, and $\sup_{s \in [0, \infty)} \xi'(s) \geq 0$.
  Moreover, $\xi$ solves the differential equation
  \begin{equation*}
    \xi^{(3)}(s) = p(s, \xi(s), \xi'(s)) + \alpha \xi''(s) \text{ for } s \geq 0,
  \end{equation*}
  where $p : [0, \infty) \times \R^2 \rightarrow \R$.
  We suppose for each $(s, q_1) \in [0, \infty) \times \R$ that the function $q_2 \mapsto p(s, q_1, q_2)$ is continuously differentiable, and $\p_{q_2} p(s, q_1, q_2) \geq \beta$.
  This condition implies for fixed $(s, q_1) \in [0, \infty) \times \R$ that the function $q_2 \mapsto p(s, q_1, q_2)$ has exactly one zero, which we denote by $q_2^\star(s, q_1)$.
  We suppose that
  \begin{equation*}
    \left\{
      \begin{aligned}
        |q_2^\star(s, q_1)| &\leq \alpha\eps_0 \text{ for } s \geq 0, |q_1| < \eps_0,
        \\
        \frac{q_2^\star(s, q_1)}{q_1} &\in [\mu, \alpha] \text{ for } s \geq 0, |q_1| \geq \eps_0.
      \end{aligned}
    \right.
  \end{equation*}
  Then, for sufficiently small $\eps_0$, $\xi'(s) \rightarrow \infty$ as $s \rightarrow \infty$.
\end{lemma}

The following proof is much clearer if one draws a diagram in the $(\xi, \xi')$-plane as they follow along.

\begin{proof}
  We set $\gamma := \frac{\mu l}{2 \alpha}$.
  By $\eps_0$ sufficiently small, we precisely mean that $\eps_0 > 0$ satisfies
  \begin{align}
    \label{eq:eps0Cond1}
    \eps_0 &< \frac{\gamma \beta^{1/2}}{2 \alpha^2}, \text{ and }
    \\
    \label{eq:eps0Cond2}
    0 &> \frac{4 \alpha^2}{\gamma \beta^{1/2}} \eps_0^2 + \eps_0 - \frac{l}{\alpha}.
  \end{align}
  Note that~\eqref{eq:eps0Cond2} implies that $\eps_0 < \frac{l}{\alpha}$.

  The conditions $\xi''(0) = 0$ and $\xi^{(3)}(0) \geq 0$ imply that $p(0, \xi(0), \xi'(0)) \geq 0$, and hence $\xi(0) \leq -l/\alpha < -\eps_0$.
  We set
  \begin{equation*}
    s_0 := \sup\left\{ \sigma \in [0, \infty) : \xi'(s) < -\gamma \text{ for all } s \in [0, \sigma) \right\}.
  \end{equation*}
  Since $\sup_{s \in [0, \infty)} \xi'(s) \geq 0$, $s_0 < \infty$.
  Note that $\xi(s_0) < \xi(0)$ and $\xi''(s_0) \geq 0$.

  Observe that for $s \in [0, \infty)$, $q_1 \leq -l/\alpha$, $q_2 \geq -\gamma$, we have
  \begin{equation*}
    p(s, q_1, q_2) \geq \gamma \beta > 0.
  \end{equation*}
  We have $\xi^{(3)}(s_0) \geq \gamma \beta > 0$, and hence there exists an $\eps > 0$ such that for $s \in (s_0, s_0 + \eps)$ we have $\xi(s) < \xi(0)$, $\xi'(s) \in \left(-\gamma, 0\right)$, and $\xi''(s) > 0$.

  We set
  \begin{equation*}
    s_1 := \sup\left\{ \sigma \in (s_0, \infty) : \xi'(s) \in \left(-\gamma, 0\right) \text{ and } \xi''(s) > 0 \text{ for all } s \in (s_0, \sigma) \right\}.
  \end{equation*}
  For $s \in [s_0, s_1)$ we have $\xi(s) < \xi(0)$ and $\xi^{(3)}(s) \geq \gamma \beta > 0$, and hence $s_1 \in (s_0, \infty)$, $\xi(s_1) < \xi(0)$, $\xi'(s_1) = 0$, and $\xi''(s_1) > 0$.

  Next, we improve the lower bound on $\xi''(s_1)$.
  Over $s \in [s_0, s_1]$ we see that $\xi''$ is non-negative and strictly monotone increasing.
  Firstly, we have
  \begin{equation}\label{eq:Xi3Bound1}
    \gamma = \xi'(s_1) - \xi'(s_0) \leq (s_1 - s_0) \xi''(s_1).
  \end{equation}
  Secondly, we have
  \begin{equation}\label{eq:Xi3Bound2}
    \xi''(s_1) \geq \gamma \beta (s_1 - s_0).
  \end{equation}
  We combine~\eqref{eq:Xi3Bound1} and~\eqref{eq:Xi3Bound2} to obtain $\xi''(s_1) \geq \gamma \beta^{1/2}$.

  Next, we set
  \begin{equation*}
    s_2 := \sup\left\{ \sigma \in [s_1, \infty) : \xi(s) < -\eps_0, \xi'(s) < 2 \alpha\eps_0, \xi''(s) > 0 \text{ for all } s \in [s_1, \sigma) \right\}.
  \end{equation*}
  For $s \in [s_1, s_2)$ we see that $p(s, \xi(s), \xi'(s)) \geq 0$, and hence $\xi^{(3)}(s) \geq \alpha \xi''(s)$.
  Therefore, $s_2 \in (s_1, \infty)$.

  We wish to show that $\xi(s_2) < -\eps_0$, and hence $\xi'(s_2) = 2\alpha\eps_0$, since $\xi''(s) \geq \gamma \beta^{1/2} > 0$ for $s \in [s_1, s_2]$.
  First, we obtain an upper bound on $s_2 - s_1$.
  We estimate
  \begin{equation*}
    \gamma \beta^{1/2} (s_2 - s_1) \leq \xi'(s_2) - \xi'(s_1) \leq 2 \alpha \eps_0.
  \end{equation*}
  Therefore, using condition~\eqref{eq:eps0Cond2}, we have
  \begin{equation*}
    \xi(s_2)
    \leq -\frac{l}{\alpha} + 2 \alpha \eps_0 (s_2 - s_1)
    \leq -\frac{l}{\alpha} + \frac{4 \alpha^2 \eps_0^2}{\gamma \beta^{1/2}}
    < -\eps_0.
  \end{equation*}
  We now have $\xi(s_2) < 0$, $\xi'(s_2) = 2 \alpha \eps_0$, and $\xi''(s_2) \geq \gamma \beta^{1/2}$.
  We set $w = \xi' - \alpha \xi$ and $z = \xi'$.
  We compute
  \begin{equation}\label{eq:PosInvSetODE1}
    \left\{
      \begin{aligned}
        z' &= w' + \alpha z,
        \\
        w'' &= p(s, \xi, \xi'),
        \\
        z(s_2) &= 2\alpha\eps_0, w(s_2) > 0, w'(s_2) \geq \gamma \beta^{1/2} - 2 \alpha^2 \eps_0.
      \end{aligned}
    \right.
  \end{equation}
  Condition~\eqref{eq:eps0Cond1} implies that $w'(s_2) > 0$.

  From our assumptions on $p$, for $z = \xi' > \alpha \eps_0$ and $w > 0$ we have $p(s, \xi, \xi') > 0$.
  We define
  \begin{equation*}
    s_3 := \sup\left\{ \sigma \in [s_2, \infty) : z(s) > \alpha \eps_0, w(s) > 0, \text{ and } w'(s) > 0 \text{ for all } s \in [s_2, \sigma) \right\}.
  \end{equation*}
  For $s \in [s_2, s_3)$,~\eqref{eq:PosInvSetODE1} implies that $z'(s) > 0$ and $w''(s) > 0$.
  Therefore, $s_3 = \infty$, and for $s \geq s_2$, $z = \xi'$ satisfies $z'(s) > \alpha z(s)$ with $z(s_2) > 0$, and hence $\xi'(s) = z(s) \rightarrow \infty$ as $s \rightarrow \infty$.
\end{proof}

Now we present the second abstract lemma.

\begin{lemma}\label{lem:UnstableSetup2}
  Let $s_0 \in \R$, $\alpha > 0$, and $u \in C^4((s_0, \infty); \R)$ which solves the  differential equation
  \begin{equation*}
    u^{(4)}(s) = \rho(u(s), u'(s), u''(s)) + \alpha u^{(3)}(s) \text{ for } s \geq s_0,
  \end{equation*}
  where $\rho \in C^1(\R^3; \R)$.
  Suppose that:
  \begin{enumerate}[(i)]
  \item
    $u(s) \rightarrow 0$ as $s \rightarrow \infty$;
  \item
    $\displaystyle \inf_{u_0, u_1, u_2 \in \R} \p_{u_2} \rho(u_0, u_1, u_2) > 0$.
    This condition implies for fixed $(u_0, u_1) \in \R^2$ that the function $u_2 \mapsto \rho(u_0, u_1, u_2)$ has exactly one zero, which we denote by $u_{2; u_0}^\star(u_1)$.
  \item\label{cond:RootBound}
    $u_{2; 0}^\star(0) = 0$, if $u_1 \neq 0$ then $\displaystyle \frac{u_{2; 0}^\star(u_1)}{u_1} \in [\mu_0, \mu_1]$, where $0 < \mu_0 \leq \mu_1 < \alpha$; and
  \item\label{cond:RootUnifConvergence}
    $u_{2; u_0}^\star \rightarrow u_{2; 0}^\star$ uniformly as $u_0 \rightarrow 0$.
  \end{enumerate}
  Then $u''(s) \rightarrow 0$ as $s \rightarrow \infty$.
\end{lemma}

\begin{proof}
  We proceed via contradiction, and hence assume that $u''(s) \not\rightarrow 0$ as $s \rightarrow \infty$.
  Observe that, since $u(s) \rightarrow 0$ as $s \rightarrow \infty$, there cannot exist a $s_1 > s_0$ and a $m > 0$ such that $|u''(s)| \geq m$ for all $s \geq s_1$.
  Therefore, there exists a $l > 0$ and a strictly monotone increasing sequence $( \sigma_i )_{i \in \N} \subset (s_0, \infty)$ which diverges to $\infty$ such that either:
  \begin{enumerate}[(A)]
  \item
    $u''(\sigma_i) \leq -l$, $u^{(3)}(\sigma_i) = 0$, $u^{(4)}(\sigma_i) \geq 0$, and $\sup_{s \geq \sigma_i} u''(s) \geq 0$ for $i \in \N$; or
  \item
    $u''(\sigma_i) \geq l$, $u^{(3)}(\sigma_i) = 0$, $u^{(4)}(\sigma_i) \leq 0$, and $\inf_{s \geq \sigma_i} u''(s) \leq 0$ for $i \in \N$.
  \end{enumerate}
  {\bf Case A:}
  Our aim is to apply \Cref{lem:UnstableSetup} to $\xi(s) = u'(s + \sigma_i)$ with a sufficiently large $i$.
  Therefore, we proceed by setting $\xi(s) := u'(s + \sigma_i)$, and we keep $i \in \N$ arbitrary for now.

  We set $p(s, q_1, q_2) := \rho(u(s + \sigma_i), q_1, q_2)$ for $(s, q_1, q_2) \in [0, \infty) \times \R^2$.
  Observe that
  \begin{equation*}
    \xi^{(3)} = p(s, \xi, \xi') + \alpha \xi'' \text{ for } s \geq 0.
  \end{equation*}
  Next, we want to confirm that $p$ satisfies the conditions of \Cref{lem:UnstableSetup}.
  Firstly, we set
  \begin{equation*}
    \beta := \inf_{u_0, u_1, u_2 \in \R} \p_{u_2} \rho(u_0, u_1, u_2) > 0.
  \end{equation*}
  We define
  \begin{align*}
    r_i &:= \max_{s \geq \sigma_i} |u(s)| \text{ and }
    \\
    \upsilon_i &:= \sup_{|u_0| \leq r_i, u_1 \in \R} \left| u_{2; u_0}^\star(u_1) - u_{2; 0}^\star(u_1) \right|.
  \end{align*}
  Observe that $r_i \rightarrow 0$ and $\upsilon_i \rightarrow 0$ as $i \rightarrow \infty$.

  For $s \in [0, \infty)$ and $q_1 \in \R$ we set $q_2^\star(s, q_1)$ to be the unique zero of $q_2 \mapsto p(s, q_1, q_2)$.
  We observe that $q_2^\star(s, q_1) = u_{2; u(s + \sigma_i)}^\star(q_1)$, and hence $|q_2^\star(s, q_1) - u_{2; 0}^\star(q_1)| \leq \upsilon_i$.
  From this and Condition~(\ref{cond:RootBound}) of this lemma, we have
  \begin{equation*}
    \left\{
      \begin{aligned}
        \frac{q_2^\star(s, q_1)}{q_1} &\in \left[ \mu_0 - \upsilon_i^{1/2}, \mu_1 + \upsilon_i^{1/2} \right] \text{ for } s \geq 0, |q_1| \geq \upsilon_i^{1/2},
        \\
        |q_2^\star(s, q_1)| &\leq \mu_1 \upsilon_i^{1/2} + \upsilon_i \text{ for } s \geq 0, |q_1| < \upsilon_i^{1/2}.
      \end{aligned}
    \right.
  \end{equation*}
  For the $\mu$ in \Cref{lem:UnstableSetup} we use $\mu := \mu_0/2$.
  Observe that for any $\eps_0 > 0$ there exists a $N \in \N$ such that $i \geq N$ implies
  \begin{equation*}
    \left\{
      \begin{aligned}
        \frac{q_2^\star(s, q_1)}{q_1} &\in \left[ \mu, \alpha \right] \text{ for } s \geq 0, |q_1| \geq \eps_0,
        \\
        |q_2^\star(s, q_1)| &\leq \alpha \, \eps_0 \text{ for } s \geq 0, |q_1| < \eps_0.
      \end{aligned}
    \right.
  \end{equation*}
  Therefore, we see that for sufficiently large $i$ our $\xi$ satisfies \Cref{lem:UnstableSetup}, and hence $\xi'(s) \rightarrow \infty$ as $s \rightarrow \infty$.
  Therefore, $u''(s) \rightarrow \infty$ as $s \rightarrow \infty$.
  However, this contradicts our assumption that $u(s) \rightarrow 0$ as $s \rightarrow \infty$, and hence $u''(s) \rightarrow 0$ as $s \rightarrow \infty$.

  {\bf Case B:}
  We set $v(s) = -u(s)$ for $s > s_0$.
  It is straightforward to check that $v$ satisfies the conditions of this lemma and those of Case A.
  Therefore, $v''(s) \rightarrow 0$ as $s \rightarrow \infty$, and hence $u''(s) \rightarrow 0$ as $s \rightarrow \infty$.
\end{proof}

We are now ready to prove \Cref{thm:BiharmonicInUnstable}.

\begin{proof}[Proof of \Cref{thm:BiharmonicInUnstable}]
  We first consider the time-reversal $u(s) = \phi(-s)$.
  We have that $u(s) \rightarrow 0$ as $s \rightarrow \infty$, and our aim is to show that $u'(s)$, $u''(s)$, and $u^{(3)}(s)$ all converge to zero as $s \rightarrow \infty$.
  We see that $u$ solves
  \begin{equation}\label{eq:FullODEReversed}
    \begin{aligned}
      u^{(4)} &= ((d-1) \cos(2 u) - (d-11) d - 21) u'' - \frac{3}{2} (d-3) (d-1) \sin(2 u)
      \\
      &\quad
      + \left(6 u''-(d-1) \sin(2 u)\right) (u')^2
      -(d-4) ((d-1) \cos (2 u)+3 d-5) u'
      \\
      &\quad
      -2 (d-4) (u')^3
      + 2 (d-4) u^{(3)}.
    \end{aligned}
  \end{equation}

  {\bf Step 1:}
  We apply \Cref{lem:UnstableSetup2} to show that $u''(s) \rightarrow 0$ as $s \rightarrow \infty$.
  We set $\alpha := 2(d - 4)$.
  The $\alpha > 0$ condition in \Cref{lem:UnstableSetup2} is what leads to our $d \geq 5$ restriction.
  We set
  \begin{align*}
    \rho(u_0, u_1, u_2)
    &:=
      - \frac{3}{2} \left(d^2 - 4 d + 3 \right) \sin(2 u_0)
      - (d-4) ((d-1) \cos(2 u_0) + 3 d - 5) u_1
    \\
    &\quad
      - (d-1) \sin(2 u_0) u_1^2
      - 2 (d - 4) u_1^3
    \\
    &\quad
      + \left(-d^2 + (d-1) \cos(2 u_0) + 11 d - 21 + 6 u_1^2 \right) u_2.
  \end{align*}
  We see that $u^{(4)} = \rho(u, u', u'') + \alpha u^{(3)}$.

  We compute, for $d \in \{5, 6, 7\}$,
  \begin{equation*}
    \p_{u_2} \rho(u_0, u_1, u_2)
    =
    (d-1) \cos(2 u_0) - (d - 11) d - 21 + 6 u_1^2
    \geq
    -d^2 + 10 d - 20 > 0
  \end{equation*}
  for $u_0$, $u_1$, $u_2 \in \R$.
  Observe that for $d \geq 8$ the polynomial $-d^2 + 10 d - 20$ is negative.
  This is what leads to our $d \leq 7$ restriction.

  We now move onto verifying Condition~(\ref{cond:RootBound}) of \Cref{lem:UnstableSetup2}.
  We rearrange $\rho = 0$ to obtain
  \begin{equation*}
    u^\star_{2; 0}(u_1) =
    2 (d-4) \, \frac{2d - 3 + u_1^2}{-d^2 + 12 d - 22 + 6 u_1^2} \, u_1.
  \end{equation*}
  We see that $u^\star_{2; 0}(0) = 0$, and for $u_1 \neq 0$
  \begin{equation*}
    \frac{u^\star_{2; 0}(u_1)}{2 (d-4) u_1}
    \in
    \left[ \frac{1}{6}, \frac{2 d - 3}{-d^2 + 12 d - 22} \right].
  \end{equation*}
  Therefore, we set
  \begin{equation*}
    \mu_0 := \frac{d-4}{3}
    \text{ and }
    \mu_1 := \frac{2 (d-4) (2 d - 3)}{-d^2 + 12 d - 22}.
  \end{equation*}
  It is straightforward to verify that $0 < \mu_0 \leq \mu_1 < \alpha$ for $d \in \{5, 6, 7\}$.

  Before we can apply \Cref{lem:UnstableSetup2}, we need to check Condition~(\ref{cond:RootUnifConvergence}) of \Cref{lem:UnstableSetup2}.
  We are interested in $u^\star_{2; u_0}(u_1)$ as $u_0 \rightarrow 0$.
  We have
  \begin{equation*}
    u^\star_{2; u_0}(u_1)
    =
    \frac{
      2 (d - 4) (2 d - 3) u_1
      + 2 (d - 4) u_1^3 + o(1) + o(u_1) + o(u_1^2)}{-d^2 + 12 d - 22 + o(1) + 6 u_1^2} \text{ as } u_0 \rightarrow 0.
  \end{equation*}
  It is straightforward to show from this that $u^\star_{2; u_0} \rightarrow u^\star_{2; 0}$ uniformly as $u_0 \rightarrow 0$.

  {\bf Step 2:}
  Since $u(s)$ and $u''(s)$ both converge to zero as $s \rightarrow \infty$, we have that $u'(s) \rightarrow 0$ as $s \rightarrow \infty$.
  All that is remaining to show is that $u^{(3)}(s) \rightarrow 0$ as $s \rightarrow \infty$.
  We proceed via contradiction, and hence assume that $u^{(3)}(s) \not\rightarrow 0$ as $s \rightarrow \infty$.
  Therefore, there exists an $\eps > 0$ and a strictly monotone increasing sequence $\sigma_i \rightarrow \infty$ such that $|u^{(3)}(\sigma_i)| \geq \eps$.
  From~\eqref{eq:FullODEReversed} we see that
  \begin{equation*}
    \p_s u^{(3)} = 2 (d - 4) u^{(3)} + o(1) \text{ as } s \rightarrow \infty.
  \end{equation*}
  Therefore, we can find a $s_\star > - s_0$ sufficiently large such that
  \begin{equation*}
    \left\{
    \begin{aligned}
      \p_s |u^{(3)}(s)| &\geq 2 (d - 4) |u^{(3)}(s)| - \eps \text{ for } s \geq s_\star, \text{ and }
      \\
      |u^{(3)}(s_\star)| &\geq \eps.
    \end{aligned}
    \right.
  \end{equation*}
  Therefore, either $u^{(3)}(s) \rightarrow -\infty$ or $u^{(3)}(s) \rightarrow \infty$ as $s \rightarrow \infty$.
  However, this contradicts with $u''(s) \rightarrow 0$ as $s \rightarrow \infty$, and hence $u^{(3)}(s) \rightarrow 0$ as $s \rightarrow \infty$.
\end{proof}

\section{Finite-time blowup once \texorpdfstring{$|\phi''|$}{|phi''|} becomes large}\label{sect:EventualBlowup}
Before we discuss the aim of this section we define, for $d \in \{5, 6, 7\}$,
\begin{equation*}
  c_\star := \max\left\{\max_{\phi \in \R} \left(-\frac{1}{12} q'(\phi)\right),\, \max_{\phi \in \R} \frac{f(\phi)}{q(\phi)},\, \max_{\phi \in \R} \left(2F(\phi)\right)^{1/2} \right\},
\end{equation*}
where $f$ and $q$ are as defined in~\eqref{eq:FullAbstractODESpecificSubstitutions}, and $F$ is from~\eqref{eq:PotentialBiharmonicODE}.
For $d = 5$, $6$, and $7$ the values of $c_\star$ are, respectively, $2 \sqrt{6}$, $3 \sqrt{5}$, and $36/\sqrt{13}$.

The aim of this section is to prove the following theorem.

\begin{theorem}\label{thm:BiharmBlowup}
  Let $d \in \{5, 6, 7\}$ and $\phi \in C^\infty((-\infty, s_f); \R)$ be an orbit in $W^u(0)$ of~\eqref{eq:FullODE}, where $s_f \in (-\infty, \infty]$ is the maximal time of existence of $\phi$.
  If there exists a $s_0 \in (-\infty, s_f)$ such that $\phi''(s_0) \geq c_\star$ (or, $\phi''(s_0) \leq -c_\star$), then $s_f < \infty$, $\phi^{(3)}(s) > 0$ (resp., $\phi^{(3)}(s) < 0$) for $s \in (s_0, s_f)$, and $\phi^{(i)}(s) \rightarrow \infty$ (resp., $\phi^{(i)}(s) \rightarrow -\infty$) as $s \nearrow s_f$ for $i \in \{0, 1, 2, 3\}$.
\end{theorem}

First we have two abstract preparatory lemmas that apply to fourth-order ODE that have the same structure as~\eqref{eq:FullODE} in the $d \in \{5, 6, 7\}$ case.
Next, we set up these two lemmas.
Let $\alpha, \beta > 0$, and $p \in C^1(\R^3; \R)$.
Let $C_1 \geq 1$ and $c_0 \geq 0$ such that
\begin{equation}\label{eq:pGrowthBounds}
  \beta (\xi_2 - c_0) \xi_1^2 + C_1^{-1} \xi_1^3
  \leq p(\xi_0, \xi_1, \xi_2) \leq
  \beta \xi_1^2 \xi_2 + C_1 \left(1 + \xi_2 + \xi_1^3\right)
\end{equation}
for $\xi_0 \in \R$, $\xi_1 \geq 0$, and $\xi_2 \geq c_0$.
In the next two lemmas we will consider solutions $u \in C^4$ to the ODE
\begin{equation}\label{eq:BlowupODE}
  u^{(4)} = p(u, u', u'') - \alpha u^{(3)}.
\end{equation}
Now we present the first preparatory lemma.

\begin{lemma}\label{lem:ThirdDerivativeGoesToInfty}
  Suppose $u \in C^4([0, s_f); \R)$ satisfies~\eqref{eq:BlowupODE}, and $u'(0)$, $u''(0) - c_0$, and $u^{(3)}(0)$ are all non-negative with at least one of them positive.
  Furthermore, suppose $s_f \in (0, \infty]$ is the maximal time of existence of $u$, that is, either $s_f = \infty$, or $s_f < \infty$ and $|(u(s), u'(s), u''(s), u^{(3)}(s))| \rightarrow \infty$ as $s \nearrow s_f$.
  Then $u^{(3)}(s) > 0$ for $s \in (0, s_f)$, and $u^{(3)}(s) \rightarrow \infty$ as $s \nearrow s_f$.
\end{lemma}

\begin{proof}
  {\bf Step 1.}
  Observe that there exists an $\eps_0 > 0$ such that $u'(s) > 0$, $u''(s) > c_0$, and $u^{(3)}(s) > 0$ for all $s \in (0, \eps_0)$.

  {\bf Step 2.}
  We show that $u^{(3)}(s) > 0$ for all $s \in (0, s_f)$.
  We define
  \begin{equation*}
    s_0 := \sup\{ \tau \in (0, s_f) : u^{(3)}(s) > 0 \text{ for all } s \in (0, \tau) \}.
  \end{equation*}
  Note that $s_0 \in (0, s_f]$.
  Observe that for all $s \in [0, s_0)$, $u$, $u'$, and $u''$ are monotone strictly-increasing, and hence
  \begin{equation}\label{eq:AbstractBlowup1}
    u^{(4)}(s) + \alpha u^{(3)}(s) \geq C_1^{-1} \left(u'(q)\right)^3 > 0 \text{ for } s \in [q, s_0),
  \end{equation}
  where $q \in (0, s_0)$.
  Therefore, after setting $q = \min\{1, s_0/2\}$ we have
  \begin{equation*}
    u^{(3)}(s) \geq \min\{u^{(3)}(q), \frac{1}{\alpha C_1} \left(u'(q)\right)^3\} > 0 \text{ for } s \in [q, s_0),
  \end{equation*}
  and hence $s_0 = s_f$.

  {\bf Step 3.}
  We show that $u^{(3)}(s) \rightarrow \infty$ as $s \nearrow s_f$.
  First we assume that $s_f = \infty$.
  From~\eqref{eq:AbstractBlowup1} we see that there exists a $s_1 > 1$ such that $u^{(3)}(s) \geq \frac{1}{2} C^{-1} \left(u'(1)\right)^3$ for all $s \geq s_1$.
  Therefore, $u''(s) \rightarrow \infty$ and $u'(s) \rightarrow \infty$, as $s \rightarrow s_f$.
  Since
  \begin{equation*}
    u^{(4)}(s) + \alpha u^{(3)}(s) \geq C_1^{-1} \left(u'(s)\right)^3 \text{ for } s \in [0, s_f),
  \end{equation*}
  we have $u^{(3)}(s) \rightarrow \infty$ as $s \rightarrow s_f$.

  Next, we suppose that $s_f < \infty$.
  Hoping for a contradiction, we assume that
  \begin{equation*}
    l_0 := \liminf_{s \nearrow s_f} u^{(3)}(s) < \infty.
  \end{equation*}
  Since $|(u(s), u'(s), u''(s), u^{(3)}(s))| \rightarrow \infty$ as $s \nearrow s_f$ and $u''$ is strictly monotone increasing, we have $u''(s) \rightarrow \infty$ as $s \nearrow s_f$.
  Therefore,
  \begin{equation*}
    \limsup_{s \nearrow s_f} u^{(3)}(s) = \infty.
  \end{equation*}
  We set
  \begin{equation*}
    g(s) := u^{(4)}(s) + \alpha u^{(3)}(s) \text{ for } s \in [0, s_f).
  \end{equation*}
  From~\eqref{eq:pGrowthBounds} we have $g(s) \rightarrow \infty$ as $s \nearrow s_f$.
  There exists a $s_2 \in [0, s_f)$ such that $g(s) > \alpha (l_0 + 1)$ for all $s \in [s_2, s_f)$ and $u^{(3)}(s_2) > l_0 + 1$.
  Therefore, $u^{(3)}(s) \geq l_0 + 1$ for all $s \in [s_2, s_f)$ which is our desired contradiction.
\end{proof}

Now we build upon this result with the following lemma.

\begin{lemma}\label{lem:ODEBlowup}
  Suppose $u \in C^4([0, s_f); \R)$ satisfies~\eqref{eq:BlowupODE}, and $u'(0)$, $u''(0) - c_0$, and $u^{(3)}(0)$ are all non-negative with at least one of them positive.
  Furthermore, suppose $s_f \in (0, \infty]$ is the maximal time of existence of $u$, that is, either $s_f = \infty$, or $s_f < \infty$ and $|(u(s), u'(s), u''(s), u^{(3)}(s))| \rightarrow \infty$ as $s \nearrow s_f$.
  Then $s_f < \infty$, and $u^{(i)}(s) \rightarrow \infty$ as $s \nearrow s_f$ for $i \in \{0, 1, 2, 3\}$.
\end{lemma}

\begin{proof}
  We set
  \begin{equation}\label{eq:RescaledVariables}
    \lambda(s) = (u^{(3)}(s))^{1/3}, v_1(s) = \frac{u'(s)}{\lambda(s)}, \text{ and } v_2(s) = \frac{u''(s)}{(\lambda(s))^2}
    \text{ for }
    s \in (0, s_f).
  \end{equation}
  Note that $\lambda(s), v_1(s), v_2(s) > 0$ for $s \in (0, s_f)$.
  We fix an arbitrary $s_0 \in (0, s_f)$.

  We compute
  \begin{equation*}
    \lambda' = \frac{1}{3 \lambda^2} \left(p(u, u', u'') - \alpha \lambda^3\right)
  \end{equation*}
  and
  \begin{equation*}
    v_1'
    =
    \lambda v_2 + \frac{\alpha v_1}{3} - \frac{v_1}{3 \lambda^3} p(u, u', u'').
  \end{equation*}
  Next, we show for very small $v_1$ and very large $\lambda$ that $v_1' > 0$.
  More precisely, for $\delta > 0$, $\lambda \geq \delta^{-1}$, and $v_1 \in (0, \delta]$, we have
  \begin{equation*}
    v_1'
    \geq
    \lambda v_2
    + \frac{\alpha}{3} v_1
    + o(\lambda v_2)
    + o(v_1)
    \text{ as }
    \delta \searrow 0.
  \end{equation*}
  Therefore, there exists a $\delta_0 > 0$ such that for $\lambda \geq \delta_0^{-1}$ and $v_1 \in (0, \delta_0]$ we have $v_1' > 0$.

  Next, we show for $\lambda \geq 1$ and very large $v_1$ that $v_1' < 0$.
  For $\delta > 0$, $\lambda \geq 1$, and $v_1 \geq \delta^{-1}$, we have
  \begin{equation*}
    v_1'
    \leq
    - \lambda v_1^3 v_2 \left(\frac{\beta}{3} + o(1)\right)
    - v_1^4 \left(\frac{1}{3 C_1} + o(1)\right)
    \text{ as }
    \delta \searrow 0.
  \end{equation*}
  Therefore, there exists a $\delta_1 > 0$ such that for $\lambda \geq 1$ and $v_1 \geq \delta_1^{-1}$, we have $v_1' < 0$.

  If $\inf_{s \in [s_0, s_f)} \lambda(s) \geq \max\{1, \delta_0^{-1}\}$ then we set $s_1 := s_0$.
  Otherwise, we set
  \begin{equation*}
    s_1 := \sup \left\{ s \in [s_0, s_f) \: \lambda(s) <  \max\{1, \delta_0^{-1}\} \right\}.
  \end{equation*}
  Note that $s_1 \in [s_0, s_f)$, since \Cref{lem:ThirdDerivativeGoesToInfty} implies that $\lambda(s) \rightarrow \infty$ as $s \nearrow s_f$.

  Putting this together, we have the following bounds on $v_1$:
  \begin{equation}\label{eq:BoundsOnV1}
    0 < a_0 := \min\{v_1(s_1), \delta_0\} \leq v_1(s) \leq \max\{v_1(s_1), \delta_1^{-1}\} =: a_1 \text{ for } s \in [s_1, s_f).
  \end{equation}
  We now consider $v_2$.
  We compute
  \begin{equation*}
    v_2'
    =
    \lambda
    + \frac{2 \alpha v_2}{3}
    - \frac{2 v_2}{3 \lambda^3} p(u, u', u'').
  \end{equation*}
  Next, we show for $\lambda \geq 1$ and very small $v_2$ that $v_2' > 0$.
  For $\delta > 0$, $\lambda \geq 1$, and $v_2 \in (0, \delta]$, we have
  \begin{equation*}
    \frac{v_2'}{\lambda}
    \geq
    1 + o(1)
    \text{ as }
    \delta \searrow 0.
  \end{equation*}
  Therefore, there exists a $\delta_2 > 0$ such that for $\lambda \geq 1$ and $v_2 \in (0, \delta_2]$ we have $v_2' > 0$.

  Next, we show for $\lambda \geq 1$ and very large $v_2$ that $v_2' < 0$.
  More precisely, for $\lambda \geq 1$, we have
  \begin{equation*}
    \frac{v_2'}{\lambda}
    \leq
    1 + C v_2 - \frac{2 \beta a_0^2}{3} v_2^2.
  \end{equation*}
  Therefore, there exists a $\Lambda_0 > 0$ such that for $v_2 \geq \Lambda_0$ we have $v_2' < 0$.

  Putting this together, we have the following bounds on $v_2$:
  \begin{equation}\label{eq:BoundsOnV2}
    0 < b_0 := \min\{v_2(s_1), \delta_2\} \leq v_2(s) \leq \min\{v_2(s_1), \Lambda_0\} =: b_1 \text{ for } s \in [s_1, s_f).
  \end{equation}
  Next, we focus on $\lambda$.
  For $\delta > 0$ and $\lambda \geq \delta^{-1}$, we have
  \begin{equation*}
    \lambda' = \lambda^2 \left(o(1) + \frac{\beta}{3} v_1^2 v_2\right)
    \text{ as }
    \delta \searrow 0.
  \end{equation*}
  Therefore, there exists a $\delta_3 > 0$ and $0 < l_0 \leq l_1$, such that for $\lambda \geq \delta_3^{-1}$ we have
  \begin{equation*}
    0 < l_0 \leq \frac{\lambda'}{\lambda^2} \leq l_1.
  \end{equation*}
  Since $\lambda(s) \rightarrow \infty$ as $s \nearrow s_f$ from \Cref{lem:ThirdDerivativeGoesToInfty}, we have $s_f < \infty$.
  Using~\eqref{eq:RescaledVariables},~\eqref{eq:BoundsOnV1}, and~\eqref{eq:BoundsOnV2}, we obtain $u^{(i)}(s) \rightarrow \infty$ as $s \nearrow s_f$ for $i \in \{1, 2, 3\}$.

  To finish up we show that $u(s) \rightarrow \infty$ as $s \nearrow s_f$.
  We define $\zeta \in C((0, s_f); \R)$ via
  \begin{equation*}
    \zeta(s) := \frac{\lambda'(s)}{(\lambda(s))^2}.
  \end{equation*}
  We set $\lambda_0 := \max\{1, \delta_0^{-1}, \delta_3^{-1}\}$.
  If $\inf_{s \in [s_0, s_f)} \lambda(s) \geq \lambda_0$ then we set $s_2 := s_0$.
  Otherwise, we set
  \begin{equation*}
    s_2 := \sup\left\{ s \in [s_0, s_f) \: \lambda(s) < \lambda_0 \right\}.
  \end{equation*}
  Note that $s_2 \in [s_1, s_f)$, since \Cref{lem:ThirdDerivativeGoesToInfty} implies that $\lambda(s) \rightarrow \infty$ as $s \nearrow s_f$, and $\zeta(s) \in [l_0, l_1]$ for $s \in [s_2, s_f)$.

  We have
  \begin{equation*}
    -\p_s \left(\frac{1}{\lambda(s)}\right) = \zeta(s) \text{ for } s \in (0, s_f),
  \end{equation*}
  and $1/\lambda(s) \rightarrow 0$ as $s \nearrow s_f$.
  For $s \in [s_2, s_f)$, we have
  \begin{equation*}
    \lambda(s) \geq \frac{1}{l_1 (s_f - s)},
  \end{equation*}
  and hence
  \begin{equation*}
    u'(s) \geq \frac{a_0}{l_1 (s_f - s)}.
  \end{equation*}
  Therefore, $u(s) \rightarrow \infty$ as $s \nearrow s_f$.
\end{proof}

We now apply \Cref{lem:ThirdDerivativeGoesToInfty} and \Cref{lem:ODEBlowup} to~\eqref{eq:FullODE} in the $d \in \{5, 6, 7\}$ case.

\begin{lemma}\label{lem:BiharmonicODEBlowup}
  Let $d \in \{5, 6, 7\}$, $s_0 \in \R$, and $\phi \in C^\infty([s_0, s_f); \R)$ be a solution of~\eqref{eq:FullODE}, where $s_f \in (-\infty, \infty]$ is the maximal time of existence of $\phi$.
  Moreover, suppose that $\phi'(s_0) > 0$, $\phi''(s_0) \geq c_\star$, and $\phi^{(3)}(s_0) \geq 0$.
  Then $s_f < \infty$, $\phi^{(3)}(s) > 0$ for $s \in (s_0, s_f)$, and $\phi^{(i)}(s) \rightarrow \infty$ as $s \nearrow s_f$ for $i \in \{0, 1, 2, 3\}$.
\end{lemma}

\begin{proof}
  For our range of $d$,~\eqref{eq:FullODE} has the same structure of~\eqref{eq:BlowupODE}.
  Specifically, we take $\alpha := 2 (d - 4)$, $\beta := 6$, $c_0 := c_\star$, and
  \begin{equation*}
    p(\xi_0, \xi_1, \xi_2) :=
    q(\xi_0) \xi_2 - f(\xi_0)
    + \beta \xi_2 \xi_1^2 + \frac{1}{2} q'(\xi_0) \xi_1^2
    + \alpha g(\xi_0) \xi_1
    + \alpha \xi_1^3,
  \end{equation*}
  where $q, f$, and $g$ are from~\eqref{eq:FullAbstractODESpecificSubstitutions}.
  Next, we verify the growth conditions~\eqref{eq:pGrowthBounds}, and hence we restrict $\xi_1 \geq 0$ and $\xi_2 \geq c_0$.
  First we consider the lower bound:
  \begin{align*}
    p(\xi_0, \xi_1, \xi_2)
    &\geq
      q(\xi_0) \left(\xi_2 - \frac{f(\xi_0)}{q(\xi_0)}\right)
      + \beta \left(\xi_2 + \frac{1}{2\beta} q'(\xi_0)\right) \xi_1^2
      + \alpha \xi_1^3
    \\
    &\geq
      q(\xi_0) \left(\xi_2 - c_0\right)
      + \beta \left(\xi_2 - c_0\right) \xi_1^2
      + \alpha \xi_1^3
    \\
    &\geq
      \beta \left(\xi_2 - c_0\right) \xi_1^2
      + \alpha \xi_1^3.
  \end{align*}
  We now consider the upper bound.
  It is straightforward to show that
  \begin{equation*}
    p(\xi_0, \xi_1, \xi_2)
    \leq
    \beta \xi_1^2 \xi_2
    + C \left(1 + \xi_2 + \xi_1^3\right).
  \end{equation*}
  This finishes our verification of~\eqref{eq:pGrowthBounds}.

  Since $\phi'(s_0) > 0$, $\phi''(s_0) \geq c_0$, and $\phi^{(3)}(s_0) \geq 0$, our conclusion follows from \Cref{lem:ThirdDerivativeGoesToInfty} and \Cref{lem:ODEBlowup}.
\end{proof}

\Cref{thm:BiharmBlowup} is a simple corollary of this for orbits in $W^u(0)$.
We prove it next.

\begin{proof}[Proof of \Cref{thm:BiharmBlowup}]
  Due to invariance of~\eqref{eq:FullODE} under $\phi \mapsto -\phi$, it suffices to only prove this theorem in the $\phi''(s_0) \geq c_\star$ case.
  Without loss of generality, we assume that $s_0$ is the first $s$ such that $\phi''(s) \geq c_\star$.
  Therefore, $\phi''(s_0) = c_\star$ and $\phi^{(3)}(s_0) \geq 0$.
  Moreover, $\Energy[\phi](s_0) > 0$ follows from $\Energy[\phi](s) \rightarrow 0$ as $s \rightarrow -\infty$ and~\eqref{eq:EnergyRateOfChange}.
  Furthermore, $U[\phi](s_0) \leq 0$, $T[\phi](s_0) > 0$, and $\phi''(s_0) > 0$, and hence $\phi'(s_0) > 0$.
  Our conclusion now follows from \Cref{lem:BiharmonicODEBlowup}.
\end{proof}

\section{Exiting \texorpdfstring{$\cC$}{C}}\label{sect:ExitingC}
In this section we restrict our attention to $d = 5$.
Before we explain the purpose of this section, we have some definitions to present.
First we define $\cF : [0, \pi] \rightarrow \R$ via
\begin{equation*}
  \cF(\phi_0) :=
  \left\{
    \begin{aligned}
      &U_0(\phi_0) \text{ for } \phi_0 \in [0, \pi/2],
      \\
      &2 \sqrt{6} \text{ for } \phi_0 \in (\pi/2, \pi],
    \end{aligned}
  \right.
\end{equation*}
where $U_0(\phi_0) = 2 \sqrt{6} \sin\phi_0$.
Now we use $\cF$ to define the central object of study in this section,
\begin{equation}\label{eq:TheMiddleRegion}
  \cC := \{ (x, y) \in \R^2 \: 0 < x < \pi \text{ and } -\cF(\pi - x) < y < \cF(x) \}.
\end{equation}
The main purpose of this section is two show that a solution to~\eqref{eq:5dODE} in $W^u(0)$ that exits $\cC$ must blowup in finite $s$-time.

Before we start, we need some set up.
For the following, we suppose that $\phi_0 \in [0, \pi/2]$.
We set $y_0 = U_0(\phi_0)$.
Note that $U(\phi_0, y_0) = 0$.

Consider the tangent line to $U_0$ at $\phi_0$, that is, $y_{\phi_0}(\phi) := U_0'(\phi_0) (\phi - \phi_0) + y_0$.
We shift the coordinate system by considering $w(s) := \phi''(s) - y_{\phi_0}(\phi(s))$, where $\phi$ solves~\eqref{eq:5dODE}.
We see that $\phi$ and $w$ satisfy the second-order system
\begin{equation}\label{eq:PhiWSystem}
  \left\{
    \begin{aligned}
      \phi''(s) &= w(s) + U'_0(\phi_0)(\phi - \phi_0) + y_0,
      \\
      w''(s) &= a(\phi_0, \phi(s), \phi'(s)) \, w(s) - 2w'(s) + P(\phi_0, \phi(s), \phi'(s)),
    \end{aligned}
  \right.
\end{equation}
where
\begin{equation}\label{eq:WOdeCoeffs}
  \left\{
    \begin{aligned}
      a &= 4 \cos(2 \phi) - 2 \sqrt{6} \cos(\phi_0) + 9 + 6 (\phi')^2, \text{ and}
      \\
      P &=
      -12 \sin(2 \phi _0) - 12 \left(\phi - \phi _0 + \sin(2 \phi) \right) + 2\sqrt{6} \left(\phi - \phi _0\right) (4 \cos(2\phi) + 9) \cos\phi_0
      \\
      &\quad
      + 12 \left(\phi _0 - \phi\right) \cos(2 \phi_0) + 2\sqrt{6}(4 \cos(2\phi) + 9) \sin\phi_0
      \\
      &\quad
      + \left(4 \cos(2\phi) - 4\sqrt{6} \cos\phi _0 + 10\right) \phi'
      \\
      &\quad
      + \left(12 \sqrt{6} \left(\sin \phi_0 + \left(\phi -\phi_0\right) \cos \phi _0 \right)- 4 \sin(2\phi)\right) (\phi')^2
      + 2 (\phi')^3.
    \end{aligned}
  \right.
\end{equation}
Observe that $P(\phi_0, \phi, v)$ is a cubic polynomial in $v$ whose coefficients are functions of $\phi_0$ and $\phi$.
We define
\begin{equation*}
  \phi_{\text{max}} :=
  \left\{
    \begin{aligned}
      &\frac{2 \sqrt{6} - y_0}{U_0'(\phi_0)} + \phi_0 \text{ for } \phi_0 \in [0, \pi/2), \text{ and}
      \\
      &\frac{\pi}{2} \text{ for } \phi_0 = \frac{\pi}{2}.
    \end{aligned}
  \right.
\end{equation*}
Note that $\phi_{\text{max}}$ is defined so that $y_{\phi_0}(\phi_{\text{max}}) = 2\sqrt{6}$.
Next, we prove some bounds on $a$ and $P$.

\begin{lemma}\label{lem:CoeffBoundsOnWODE}
  We have the following bounds on $a$ and $P$ as defined in~\eqref{eq:WOdeCoeffs}:
  \begin{enumerate}[(i)]
  \item
    $a(\phi_0, \phi, v) \geq 0.1$ for all $\phi_0, \phi, v \in \R$; and
  \item
    $P(\phi_0, \phi, v) \geq 0$ for $0 \leq \phi_0 \leq \pi/2$, $\phi_0 \leq \phi \leq \phi_{\text{max}}$, and $v \geq 0$, with equality only when $\phi_0$, $\phi$, and $v$ are all zero.
  \end{enumerate}
\end{lemma}

The following proof is computer-assisted using interval arithmetic, see, for example,~\cite{AlefeldHerzbergerInterval},~\cite{KulischMirankerInterval},~\cite{Moore66Interval}, or~\cite{Moore79Interval}.
We outline, at a high-level, the steps and intermediate results of the computations.

\begin{proof}
  Part (i) easily follows from the expression for $a$ in~\eqref{eq:WOdeCoeffs}.

  Now we focus on Part (ii).
  We first observe that $P(0, 0, v) = (14 - 4\sqrt{6}) v + 2 v^3$, and hence $P(0, 0, 0) = 0$.
  In what follows, we view $P(\phi_0, \phi, v)$ as a cubic polynomial in $v$, and analyse the coefficients separately.

  {\bf Coefficient of $v^0$}.
  Using a standard branch-and-bound algorithm for interval arithmetic, we show that this coefficient is bounded below by $0.01$ for
  \begin{equation*}
    0.4 \leq \phi_0 \leq \frac{\pi}{2} \text{ and } 0 \leq \phi \leq \frac{\pi}{2}.
  \end{equation*}
  Therefore, we now only need to consider the case where $\phi_0 < 0.4$.
  We introduce the new variable
  \begin{equation*}
    z := \frac{\phi - \phi_0}{\phi_{\text{max}} - \phi_0}.
  \end{equation*}
  Note that the condition $\phi_0 \leq \phi \leq \phi_{\text{max}}$ is equivalent to $0 \leq z \leq 1$.
  Working in $(\phi_0, z)$-coordinates we show, separately, that this coefficient is bounded below by $0.01$ for the two regions
  \begin{equation*}
    0.01 \leq \phi_0 \leq 0.4 \text{ and } 0 \leq z \leq 1,
  \end{equation*}
  and
  \begin{equation*}
    0 \leq \phi_0 \leq 0.4 \text{ and } 0.01 \leq z \leq 1.
  \end{equation*}
  Now, we focus on the region
  \begin{equation}\label{eq:Phi0ZBounbdsAtOrigin}
    0 \leq \phi_0 \leq 0.01 \text{ and } 0 \leq z \leq 0.01.
  \end{equation}
  For this we switch back to $(\phi_0, \phi)$-coordinates.
  The bounds~\eqref{eq:Phi0ZBounbdsAtOrigin} imply $0 \leq \phi \leq 0.021$.
  We use Taylor expansion around zero of $\sin$ up to sixth-order and $\cos$ up to fifth-order, and enclose the remainder terms in intervals.
  After this expansion, we compute the following enclosure of this coefficient over the region $0 \leq \phi_0 \leq 0.01$ and $0 \leq \phi \leq 0.021$,
  \begin{equation*}
    [13.2121, 13.24] \phi_0^3
    +
    [15.673, 15.6867] \phi.
  \end{equation*}
  Therefore, this coefficient is strictly positive over the range of $\phi_0$ and $\phi$ supposed in this lemma, except when $\phi_0$ and $\phi$ are both zero.

  {\bf The coefficient of $v^2$.}
  We show that this coefficient is non-negative for $0 \leq \phi_0 \leq \pi/2$ and $0 \leq \phi \leq \pi/2$, and is zero only when $\phi_0 = \phi = 0$.
  Using the same standard branch-and-bound algorithm for interval arithmetic, we show, separately, that this coefficient is bounded below by $0.01$ over the regions
  \begin{equation*}
    0.11 \leq \phi_0 \leq \frac{\pi}{2} \text{ and } 0 \leq \phi \leq \frac{\pi}{2},
  \end{equation*}
  and
  \begin{equation*}
    0 \leq \phi_0 \leq \frac{\pi}{2} \text{ and } 0.0006 \leq \phi \leq \frac{\pi}{2}.
  \end{equation*}
  It now suffices to consider the region
  \begin{equation*}
    0 \leq \phi_0 \leq 0.11 \text{ and } 0 \leq \phi \leq 0.0006.
  \end{equation*}
  In exactly the same way as for the $v^0$ coefficient, we expand $\cos$ and $\sin$ into their sixth- and fifth-order, respectively, Taylor series around zero, and enclose the remainder terms in intervals.
  Using this expansion, we compute the following enclosure of this coefficient over this region,
  \begin{equation*}
    [9.76536, 9.83056] \phi_0^3 + [21.2159, 21.4586] \phi.
  \end{equation*}
  Therefore, the polynomial $v \mapsto P(\phi_0, \phi, v)$ is convex, and strictly convex when $(\phi_0, \phi) \neq 0$.

  {\bf The coefficient of $v^1$.}
  If $\phi_0$ and $\phi$ are such that this coefficient is non-negative then the above computations yield the conclusion of this lemma.
  Using the same standard branch-and-bound algorithm for interval arithmetic, we show that this coefficient is bounded below by $0.01$ in the region
  \begin{equation*}
    1 \leq \phi_0 \leq \frac{\pi}{2} \text{ and } 0 \leq \phi \leq \frac{\pi}{2}.
  \end{equation*}
  Therefore, we may restrict our attention to $\phi_0 < 1$.

  Now, similarly to when we were considering the coefficient of $v^0$, we work in $(\phi_0, z)$-coordinates.
  Using a standard divide-and-conquer algorithm, we compute the following enclosure of the subset of $[0, 1] \times [0, 1]$ in $(\phi_0, z)$-coordinates in which the $v^1$-coefficient is bounded above by $0.01$,
  \begin{equation*}
    A := [0, 783/1024] \times [779/1024, 1],
  \end{equation*}
  and hence we focus our attention here.

  Now we drop the $2 v^3$ term from $P(\phi_0, \phi, v)$, and estimate the minimum of the remaining quadratic in $v$, after fixing $(\phi_0, \phi) \in A$, using the same standard branch-and-bound algorithm for interval arithmetic used above.
  Note that this forms a lower bound of $P(\phi_0, \phi, v)$ for the values of the parameters of interest to us.
  These computations show that for $(\phi_0, \phi) \in A$ the minimum value of this quadratic in $v$ is greater than or equal to $0.5$.
\end{proof}

Next, we show that if an orbit in $W^u(0)$ exits $\cC$ through the top-left or bottom-right, then it must blowup in finite $s$-time.

\begin{lemma}\label{lem:LeftSideExit}
  Let $\phi \in C^\infty((-\infty, s_f); \R)$ be a solution to~\eqref{eq:5dODE}, where $s_f \in (-\infty, \infty]$ is the maximal time of existence of $\phi$.
  Suppose that $\Energy[\phi](s) > 0$ for all $s \in (-\infty, s_f)$.
  Moreover, suppose that there exists $s_0 < s_1 < s_f$ such that $(\phi(s), \phi''(s)) \in \cC$ for all $s \in [s_0, s_1)$, $(\phi(s_1), \phi''(s_1)) \in \p\cC$, and either:
  \begin{enumerate}[(i)]
  \item
    $\phi(s_1) \in [0, \pi/2]$, and $\phi''(s_1) \geq 0$; or
  \item
    $\phi(s_1) \in [\pi/2, \pi]$, and $\phi''(s_1) \leq 0$.
  \end{enumerate}
  Then there exists a $s_\star \in [s_1, s_f)$ such that $\phi''(s_\star) = 2\sqrt{6}$ in the case of (i), or $\phi''(s_\star) = -2\sqrt{6}$ in the case of (ii).
  Therefore, $s_f < \infty$ by \Cref{thm:BiharmBlowup}.
\end{lemma}

\begin{proof}
  Due the symmetry of our situation under the transformation $\phi \mapsto \pi - \phi$, we assume without loss of generality that we are in the situation described by (i).

  We have $\phi''(s_1) = U_0(\phi(s_1))$.
  Using the set up from the beginning of this section, we set $\phi_0 = \phi(s_1)$ and $y = U_0(\phi_1)$.
  If $\phi_1 = \pi/2$ then $\phi''(s_1) = 2\sqrt{6}$, and we are done.
  Therefore, it suffices to focus on the case $0 \leq \phi_0 < \pi/2$.

  We work with the $(\phi, w)$-system as defined in~\eqref{eq:PhiWSystem}.
  Observe that $w(s_1) = 0$, and $w'(s_1) \geq 0$, since $w(s_1) < 0$ for all $s \in [s_0, s_1)$.
  From~\eqref{eq:EnergySplit}, we see that, in our context,
  \begin{equation*}
    T[\phi] = \phi' \left(\phi^{(3)} + 2 \phi'' - \frac{1}{2} (4 \cos(2 \phi) + 9) \phi' - \frac{3}{2} (\phi')^3\right).
  \end{equation*}
  Because $\phi$ cannot be trivial, we know that $T[\phi](s_1) > 0$, since $\Energy[\phi](s_1) > 0$ and $U[\phi](s_1) = 0$.
  Moreover, $w'(s_1) \geq 0$ implies $\phi^{(3)}(s_1) \geq U_0'(\phi_0) \phi'(s_1)$.
  We conclude from these two inequalities that $\phi'(s_1) > 0$.

  From~\eqref{eq:PhiWSystem} and \Cref{lem:CoeffBoundsOnWODE}, we can find a $s_2 \in (s_1, s_f)$ such that $\phi'(s_2), w(s_2), w'(s_2) > 0$ and $\phi(s_2) \in (\phi_0, \phi_{\text{max}})$.
  Hoping for a contradiction, we assume that $\phi''(s) < 2\sqrt{6}$ for all $s \in (s_1, s_f)$.
  Therefore, if $w(s) \geq 0$ then $\phi(s) < \phi_{\text{max}}$.

  From~\eqref{eq:PhiWSystem} and \Cref{lem:CoeffBoundsOnWODE}, we see that if $\phi', w, w' > 0$ and $\phi \in (\phi_0, \phi_{\text{max}})$, then $\phi'' > 0$ and
  \begin{equation}
    w'' > c_1 w - 2w',
  \end{equation}
  where $c_1 > 0$.
  Next, we see that $w(s), w'(s), \phi'(s), \phi''(s) > 0$ for $s \in [s_2, s_f)$.
  Therefore, $|\phi''|$ is bounded, and hence $s_f = \infty$.

  Next, observe that $\phi'(s) \geq \phi'(s_2) > 0$ for $s \geq s_2$.
  This is our desired contradiction, since $\phi(s) < \phi_{\text{max}}$ cannot be true for all $s \in [s_1, s_f)$.
\end{proof}

Next, we show that there is another useful positive invariant cone for~\eqref{eq:5dODE}.
We consider the convex cone
\begin{equation*}
  \cK := \left\{ (x, y) \in \R^2 \:  x \geq 0 \text{ and } y \geq 3 x \right\}.
\end{equation*}
We are interested in showing that the condition $(\phi, \phi''), (\phi', \phi^{(3)}) \in \cK$ is positive invariant under the flow of~\eqref{eq:5dODE}.
We set $\xi := \phi'' - 3\phi$, and with this $(\phi, \phi''), (\phi', \phi^{(3)}) \in \cK$ is equivalent to $\phi, \phi', \xi, \xi' \geq 0$.

From~\eqref{eq:5dODE} we compute
\begin{equation*}
  \xi''
  =
  \left(6 \left(\phi'\right)^2 + 4 \cos(2 \phi) + 6\right) \xi - 2 \xi' + Q(\phi, \phi'),
\end{equation*}
where
\begin{equation*}
  Q(\phi, v) := 2 v^3 + (18\phi - 4 \sin(2 \phi)) v^2 + 8 \cos^2(\phi) v + 6 (3 \phi - 2 \sin(2 \phi) + 2 \phi \cos(2\phi)).
\end{equation*}
We now write the fourth-order ODE~\eqref{eq:5dODE} as a system of two second-order ODE
\begin{equation}\label{eq:SecondOrderDecomposition}
  \left\{
    \begin{aligned}
      \phi'' &= 3\phi + \xi,
      \\
      \xi'' &= \left(6 \left(\phi'\right)^2 + 4 \cos(2 \phi) + 6\right) \xi - 2 \xi' + Q(\phi, \phi').
    \end{aligned}
  \right.
\end{equation}
Observe that if $\xi \geq 0$ then $\xi'' \geq 2\xi - 2\xi' + Q(\phi, \phi')$.

\begin{lemma}\label{lem:PositiveInvariantCone}
  Let $s_0 < s_1$, $\phi \in C^\infty([s_0, s_1]; \R)$ be a solution to~\eqref{eq:5dODE}, and $\xi = \phi'' - 3\phi$.
  Moreover, suppose that $\phi(s_0)$, $\phi'(s_0)$, $\xi(s_0)$, and $\xi'(s_0)$ are all non-negative.
  Then $\phi(s)$, $\phi'(s)$, $\xi(s)$, $\xi'(s)$ are all non-negative for $s \in [s_0, s_1]$.
  Moreover, if one of $\phi(s_0)$, $\phi'(s_0)$, $\xi(s_0)$, or $\xi'(s_0)$ is positive, then $\phi(s)$, $\phi'(s)$, $\xi(s)$, $\xi'(s)$ are all positive for $s \in (s_0, s_1]$.
\end{lemma}

\begin{proof}
  First we show that $Q(\phi, v) \geq 0$ for $\phi, v \geq 0$ with equality only when $(\phi, v) = 0$.
  We can consider $Q(\phi, v)$ as a cubic in $v$ with coefficients being functions of $\phi$.
  Clearly $Q(0, 0) = 0$, and the coefficients of $v^1$, $v^2$, and $v^3$ in $Q(\phi, v)$, are all non-negative for $\phi \geq 0$.

  Next, we show that the coefficient of $v^0$ is positive for $\phi > 0$.
  For $\phi \in [0, \pi/8]$ we have
  \begin{equation*}
    6 (3 \phi - 2 \sin(2 \phi) + 2 \phi \cos(2\phi))
    \geq 6 (\sqrt{2} - 1) \phi.
  \end{equation*}
  We also have the trivial lower bound, for $\phi \geq 0$,
  \begin{equation*}
    6 (3 \phi - 2 \sin(2 \phi) + 2 \phi \cos(2\phi))
    \geq 6 (\phi - 2),
  \end{equation*}
  and hence we have this bounded uniformly above zero for, say, $\phi \geq 3$.
  Next, we use the same branch-and-bound algorithm for interval arithmetic that we used multiple times in the proof of \Cref{lem:CoeffBoundsOnWODE} to verify that
  \begin{equation*}
    \min_{\phi \in [\pi/8, 3]} 6 (3 \phi - 2 \sin(2 \phi) + 2 \phi \cos(2\phi)) > 1.9,
  \end{equation*}
  and hence the coefficient of $v^0$ is positive for $\phi > 0$.

  Observe that the condition $\phi, \phi', \xi, \xi' > 0$ is a positive invariant under the flow of~\eqref{eq:SecondOrderDecomposition}.
  Moreover, if $\phi(s_0), \phi'(s_0), \xi(s_0), \xi'(s_0) \geq 0$ and one of these are positive, then it easy to see from~\eqref{eq:SecondOrderDecomposition} that $\phi(s), \phi'(s), \xi(s), \xi'(s) > 0$ for $s \in (s_0, s_1]$.
  If $\phi, \phi', \xi, \xi' = 0$, then uniqueness of solutions implies that $\phi, \xi \equiv 0$.
  Putting this together yields our claim.
\end{proof}

Next, we use this lemma to prove that once a non-trivial orbit enters the $(\phi, \phi''), (\phi', \phi^{(3)}) \in \cK$ cone then it must blowup in finite $s$-time.

\begin{lemma}\label{lem:ConeBlowup}
  Let $s_0 \in \R$, $\phi \in C^\infty([s_0, s_f); \R)$ solves~\eqref{eq:5dODE}, where $s_f \in (s_0, \infty]$ is the maximal time of existence of $\phi$.
  Set $\xi = \phi'' - 3\phi$, and suppose $\phi(s_0)$, $\phi'(s_0)$, $\xi(s_0)$, $\xi'(s_0)$ are all non-negative with at least one of them positive.
  Then there exists a $s_1 \in [s_0, s_f)$ such that $\phi'(s_1) > 0$, $\phi''(s_1) \geq 2\sqrt{6}$ and $\phi^{(3)}(s_1) > 0$, and hence $s_f < \infty$ by \Cref{lem:BiharmonicODEBlowup}.
\end{lemma}

\begin{proof}
  From \Cref{lem:PositiveInvariantCone}, we know that $\phi(s)$, $\phi'(s)$, $\xi(s)$, $\xi'(s)$ are all positive, and hence $\phi''(s)$ and $\phi^{(3)}(s)$ are also positive, for $s \in (s_0, s_f)$.
  If there exists a $s \in (s_0, s_f)$ such that $\phi''(s) \geq 2\sqrt{6}$ then we are done.

  Hoping for a contradiction, we assume that $\phi''(s) < 2\sqrt{6}$ for all $s \in (s_0, s_f)$, and hence $s_f = \infty$, since $|\phi''|$ is bounded.
  We have $\xi''(s) \geq 2\xi(s) - 2\xi'(s)$ on $[s_0, \infty)$, and hence $\xi(s)$ and $\xi'(s)$ both diverge to infinity as $s \rightarrow \infty$.
  From~\eqref{eq:SecondOrderDecomposition} we see that $\phi(s)$ and $\phi'(s)$ both diverge to infinity as $s \rightarrow \infty$.
  Putting all of this together, we see that there exists a $s_1 > s_0$ such that $\phi''(s_1) \geq 2\sqrt{6}$ which is our desired contradiction.
\end{proof}

Next, we show that if an orbit in $W^u(0)$ exits $\cC$ through the left or right, then it must blowup in finite $s$-time.

\begin{lemma}\label{lem:RightSideExit}
  Let $\phi \in C^\infty((-\infty, s_f); \R)$ be a solution of~\eqref{eq:5dODE}, where $s_f \in (-\infty, \infty]$ is the maximal time of existence of $\phi$.
  Suppose that $\Energy[\phi](s) > 0$ for all $s \in (-\infty, s_f)$.
  Furthermore, suppose that there exists $s_0 < s_1 < s_f$ such that $(\phi(s), \phi''(s)) \in \cC$ for $s \in [s_0, s_1)$, and $\phi(s_1) = \pi$ (or, $\phi(s_1) = 0$).
  Then there exists a $s_\star \in [s_1, s_f)$ such that $\phi''(s_\star) = 2\sqrt{6}$ (resp., $\phi(s_\star) = -2\sqrt{6}$), and hence $s_f < \infty$ by \Cref{thm:BiharmBlowup}.
\end{lemma}

\begin{proof}
  Due the symmetry of our situation under the transformation $\phi \mapsto \pi - \phi$, we assume without loss of generality that we are in the situation where $\phi(s_1) = \pi$.

  We set $\nu = \phi - \pi$ and $\xi = \nu'' - 3 \nu$.
  Observe that $\nu$ also solves~\eqref{eq:5dODE}, and $(\nu, \xi)$ satisfies~\eqref{eq:SecondOrderDecomposition}.
  Since $(\phi(s_1), \phi''(s_1)) \in \p\cC$ and $\phi(s_1) = \pi$, we have $\phi''(s_1) \geq 0$, and hence $\nu''(s_1) \geq 0$.
  Now observe that $\nu(s_1) = 0$ and $\xi(s_1) \geq 0$.
  Next, we show that $\nu'(s_1) = \phi'(s_1)  > 0$ and $\xi'(s_1) > 0$.
  After this the conclusion follows from \Cref{lem:ConeBlowup}.

  We know that $\Energy[\phi](s_1) > 0$ and $U[\phi](s_1) = - \frac{1}{2} (\phi''(s_1))^2$.
  Therefore, $T[\phi](s_1) > \frac{1}{2} (\phi''(s_1))^2 \geq 0$.
  We know that $\phi'(s_1) > 0$, since $s_1$ is the exit time from $\cC$ and $T[\phi](s_1) > 0$.

  Now we turn our attention to showing that $\xi'(s_1) = \phi^{(3)}(s_1) - 3\phi'(s_1) > 0$.
  After making the substitutions $\phi'(s_1) \mapsto v$ and $\phi''(s_1) \mapsto y$, the inequalities $\phi'(s_1) > 0$ and $T[\phi](s_1) > 0$ give the following lower bound on $\phi^{(3)}(s_1)$,
  \begin{equation*}
    \phi^{(3)}(s_1) > \frac{y^2}{2v} - 2y + \frac{13}{2}v + \frac{3}{2} v^3.
  \end{equation*}
  Therefore, the task of showing that $\xi'(s_1) > 0$ reduces to showing
  \begin{equation*}
    y^2 - 4 y v + 7 v^2 + 3 v^4 \geq 0 \text{ for all } v > 0 \text{ and } y \geq 0.
  \end{equation*}
  This is easy to show, since we can write left hand side as a sum of squares
  \begin{equation*}
    y^2 - 4 y v + 7 v^2 + 3 v^4 = 3 v^4 + 7 \left(v - \frac{2}{7} y\right)^2 + \frac{3}{7} y^2.\qedhere
  \end{equation*}
\end{proof}

\begin{remark}
  Note that solutions of~\eqref{eq:5dODE} in $W^u(0)$ cannot exit $\cC$ with $(\phi, \phi'') = 0$ or $(\phi, \phi'') = (\pi, 0)$.
  Indeed, \Cref{lem:LeftSideExit} and \Cref{lem:RightSideExit} apply to both of these points, with one of these lemmas giving a later $s$-time at which $\phi'' = 2\sqrt{6}$, and the other a later $s$-time at which $\phi'' = -2\sqrt{6}$.
  \Cref{thm:BiharmBlowup} implies that only one of these is possible, and this yields a contradiction.

  We can show this previous statement more directly by noting that upon exiting through one of these points $T[\phi] > 0$.
  The restriction this puts on $(\phi', \phi^{(3)})$ is not compatible with this being a $s$-time at which the orbit exits $\cC$.
  \hfill$\blacklozenge$
\end{remark}

\begin{remark}
  One of the barriers to extending all of our results to $d \in \{6, 7\}$ is that \Cref{lem:RightSideExit} is not true when $d \in \{6, 7\}$.
  Indeed, for $d \in \{6, 7\}$, $\phi' > 0$ and $T[\phi] > 0$ does not even imply that $\phi^{(3)} \geq 0$.
  This makes it impossible to find a suitable positive invariant cone analogous to $\cK$.
  \hfill$\blacklozenge$
\end{remark}

Combining \Cref{thm:BiharmBlowup}, \Cref{lem:LeftSideExit}, and \Cref{lem:RightSideExit}, we have the following corollary.

\begin{corollary}\label{cor:ExitMiddleBlowup}
  Let $\phi \in C^\infty((-\infty, s_f); \R)$ be an orbit in $W^u(0)$ of~\eqref{eq:5dODE}, where $s_f \in (-\infty, \infty]$ is the maximal time of existence of $\phi$.
  Suppose that there exists $s_0 < s_1 < s_f$ such that $(\phi(s), \phi''(s)) \in \cC$ for all $s \in [s_0, s_1)$, and $(\phi(s_1), \phi''(s_1)) \in \p \cC$.
  Then there exists a $s_\star \in [s_1, s_f)$ such that $\phi''(s_\star) = 2\sqrt{6}$ (or, $\phi''(s_\star) = -2\sqrt{6}$), and hence $s_f < \infty$ and $\phi^{(i)}(s) \rightarrow \infty$ (resp., $\phi^{(i)}(s) \rightarrow -\infty$) as $s \nearrow s_f$ for $i \in \{0, 1, 2, 3\}$.
\end{corollary}

\section{Existence of heteroclinic orbits}\label{sect:HeteroOrbitsExist}
Again, in this section we focus on the $d = 5$ case, except for some preparatory lemmas in which we can relax this condition on $d$.
In the last section we considered what happened to orbits in $W^u(0)$ which exited $\cC$.
In this section we prove that non-trivial orbits in $W^u(0)$ either stay in $\cC \cup -\cC$ forever, or blowup in finite $s$-time.
Furthermore, we show that the orbits which stay in $\cC$ forever are heteroclinic orbits connecting the origin and $(\pi/2, 0, 0, 0)$, and we prove the existence of such an orbit.
By combining these results we end with a proof of \Cref{thm:EntireSolutionsAreHeteroclinic}.
We begin with some preparatory lemmas.

\begin{lemma}\label{lem:BoundedQBoundedQDot}
  Let $d \in \N$, $R > 0$, and $\phi \in C^\infty((-\infty, s_1]; \R)$ be a solution of~\eqref{eq:FullODE}, where $s_1 \in \R$.
  Suppose that $|(\phi(s), \phi''(s))| \leq R$ for all $s \in (-\infty, s_1]$.
  Then there exists $C > 0$, depending only on $d$ and $R$ such that,
  \begin{equation*}
    |(\phi'(s), \phi^{(3)}(s))| \leq C \text{ for } s \leq s_1.
  \end{equation*}
\end{lemma}

\begin{proof}
  First we focus on showing the boundedness of $\phi'$.
  Since $|\phi''(s)| \leq R$ for all $s \leq s_1$, we must have $|\phi'(s)| \leq C(R)$, or else $|\phi(s)| \leq R$ will not be true for all $s \leq s_1$.

  Next, we use this with~\eqref{eq:FullODE} to obtain
  \begin{equation*}
    |\p_s \phi^{(3)} + 2 (d-4) \phi^{(3)}| \leq C(d, R).
  \end{equation*}
  Since $|\phi''(s)| \leq R$ for all $s \leq s_1$, we have $|\phi^{(3)}(s)| \leq C(d, R)$ for all $s \leq s_1$.
\end{proof}

Next, we have another preparatory lemma.

\begin{lemma}\label{lem:LimitWhenBoundedForAllTime}
  Let $d \in \N_{\geq 5}$ and $\phi \in C^\infty(\R; \R)$ be a non-trivial solution to~\eqref{eq:FullODE} in $W^u(0)$.
  Suppose
  \begin{equation*}
    \mathcal{A} := \left\{ (\phi(s), \phi''(s)) \: s \in \R \right\} \subset \R^2
  \end{equation*}
  is bounded.
  Then $(\phi(s), \phi'(s), \phi''(s), \phi^{(3)}(s)) \rightarrow (\phi_\infty, 0, 0, 0)$, where $\phi_\infty = (k + \frac{1}{2}) \pi$ for some $k \in \Z$.
\end{lemma}

\begin{proof}
  From~\eqref{eq:EnergyRateOfChange}, we estimate
  \begin{equation*}
    \p_s \Energy[\phi] \geq 2 (d - 4) \left( (\phi'')^2 + (d - 2) (\phi')^2 + (\phi')^4 \right).
  \end{equation*}
  Since $\phi$ is non-trivial, we have $\Energy[\phi](s) > 0$ for all $s \in \R$.

  Next, we show that $\Energy[\phi]$ must be bounded.
  We know that $U[\phi]$ is bounded from above.
  Therefore, if $\Energy[\phi]$ was not bounded then for any $T_0 > 0$ we would be able to find a $s_0 \in \R$ such that $T[\phi](s) \geq T_0$ for all $s \geq s_0$.
  By taking $T_0$ sufficiently large we can ensure that $|(\phi', \phi^{(3)})|$ is as large as we like.
  \Cref{lem:BoundedQBoundedQDot} shows that this is not possible, and hence $\Energy[\phi]$ is monotone and bounded.

  \Cref{lem:BoundedQBoundedQDot} implies that $\phi'$ and $\phi^{(3)}$ are bounded on $\R$.
  We compute
  \begin{align*}
    \p_s^2 \Energy[\phi]
    &
      = 2 (d-4) ( 2 \phi'' \phi^{(3)} - (d-1) \sin(2 \phi) (\phi')^3 + ((d-1) \cos(2 \phi) + 3 d - 5) \phi' \phi''
    \\
    &\quad
      + 4 (\phi')^3 \phi'' ),
  \end{align*}
  and hence $|\p_s^2 \Energy[\phi](s)| \leq C$ for all $s \in \R$.
  Therefore, $\p_s \Energy[\phi](s) \rightarrow 0$, and hence $\phi'(s) \rightarrow 0$ and $\phi''(s) \rightarrow 0$, as $s \rightarrow \infty$.
  From~\eqref{eq:FullODE}, we see that $\phi^{(4)}$ is bounded on $\R$, and hence $\phi^{(3)}(s) \rightarrow 0$ as $s \rightarrow \infty$.
  Differentiating~\eqref{eq:FullODE} we see that $\phi^{(5)}$ is also bounded, and hence $\phi^{(4)}(s) \rightarrow 0$ for $s \rightarrow \infty$.
  Using these facts with~\eqref{eq:FullODE}, we see that $\sin(2 \phi(s)) \rightarrow 0$ as $s \rightarrow \infty$.
  Therefore, we see that there is a $l \in\Z$ such that $\phi(s) \rightarrow \frac{\pi}{2}l$ as $s \rightarrow \infty$.
  However, since we know that $\Energy[\phi](s) > 0$ for $s \in \R$, $l$ must be odd.
\end{proof}

Before we continue we want to collect some facts about the structure of $W^u(0)$ locally around the origin.
For this little detour, we assume that $d \in \N_{\geq 4}$.

From the stable manifold theorem, we know that $W^u(0) \subset \R^4$ is a two-dimensional $C^\infty$-manifold, and that its tangent plane at the origin is spanned by $(1, 1, 1, 1)^T$ and $(1, 3, 9, 27)^T$.
Since~\eqref{eq:FullODE} is invariant under $\phi \mapsto -\phi$, $W^u(0)$ is symmetric under reflection through the origin.
Locally around origin, we can parameterise $W^u(0)$ by
\begin{equation}\label{eq:W0LocalParam}
  \Gamma : z \mapsto \gamma_1(z) \eta_1 + \gamma_2(z) \eta_2 + z_1 \eta_3 + z_2 \eta_4 \text{ for } z := (z_1, z_2) \in \mathcal{U},
\end{equation}
where:
\begin{itemize}
\item
  $\mathcal{U} \subset \R^2$ is an open ball centred at the origin;
\item
  $\gamma := (\gamma_1, \gamma_2) \in C^\infty(\mathcal{U}; \R^2)$, $\gamma(0) = 0$, $D\gamma(0) = 0$, and $\gamma(-z) = -\gamma(z)$ for $z \in \mathcal{U}$; and
\item
  $\eta_i$, for $i \in \{1, 2, 3, 4\}$, are the eigenvectors of the linearization of~\eqref{eq:FullODE}, with $\eta_1$ and $\eta_2$ corresponding to the negative eigenvalues, and $\eta_3 = (1, 1, 1, 1)^T$ and $\eta_4 = (1, 3, 9, 27)^T$ correspond to the positive eigenvalues one and three, respectively.
\end{itemize}
By choosing $\eps > 0$ sufficiently small, if $(\phi(s), \phi'(s), \phi''(s), \phi^{(3)}(s)) = \Gamma(z(s))$, where $|z(s)| < \eps$ and $\phi$ solves~\eqref{eq:FullODE}, then
\begin{equation}\label{eq:W0ReducedODE}
  \p_s z =
  \begin{pmatrix}
    1 & 0 \\
    0 & 3
  \end{pmatrix}
  z
  + \cO(|z|^3) \text{ as } z \rightarrow 0.
\end{equation}
We can see that, locally around the origin, $W^u(0)$ is a graph over the $(\phi, \phi'')$-plane.
More precisely, from the inverse function theorem, there exists a $\mathcal{W} := (\mathcal{W}_1, \mathcal{W}_2) \in C^1(\mathcal{U}'; \R^2)$, where $\mathcal{U}' \subset \R^2$ is an open ball centred at the origin, such that locally around the origin we can parameterise $W^u(0)$ by
\begin{equation}\label{eq:W0LocalGraphParam}
  z \mapsto (z_1, \mathcal{W}_1(z), z_2, \mathcal{W}_2(z)) \text{ for } z := (z_1, z_2) \in \mathcal{U}'.
\end{equation}
Furthermore, $\mathcal{W}(0) = 0$, $\mathcal{W}(-z) = -\mathcal{W}(z)$ for $z \in \mathcal{U}'$, and
\begin{equation}\label{eq:DW0}
  D\mathcal{W}(0) = \frac{1}{4}
  \begin{pmatrix}
    3 & 1 \\
    -9 & 13
  \end{pmatrix}.
\end{equation}
We use these facts in our next lemma, and shift our focus back to $d = 5$.

\begin{lemma}\label{lem:OrbitsInLocalUnstableManifold}
  Suppose that $\phi \in C^\infty((-\infty, s_f); \R)$ is a non-trivial solution to~\eqref{eq:5dODE} in $W^u(0)$, where $s_f \in (-\infty, \infty]$ is its maximal time of existence.
  Then there exists a $s_0 < s_f$ such that either:
  \begin{enumerate}[(i)]
  \item
    $(\phi(s), \phi''(s)) \in \cC \cup -\cC$ for all $s < s_0$; or
  \item\label{item:AlwaysOutsideC}
    $(\phi(s), \phi''(s)) \in \R^2 \setminus \overline{(\cC \cup -\cC)}$ for all $s < s_0$.
  \end{enumerate}
  Furthermore, if we are in the situation described by~(\ref{item:AlwaysOutsideC}), then there exists a $s_\star \in (-\infty, s_f)$ such that $|\phi''(s_\star)| = 2 \sqrt{6}$, and hence $s_f < \infty$ by \Cref{thm:BiharmBlowup}.
\end{lemma}

\begin{proof}
  Combining~\eqref{eq:W0LocalGraphParam} and~\eqref{eq:DW0}, we have, for sufficiently small $|(\phi, \phi'')|$,
  \begin{equation}\label{eq:ReducedODE}
    \p_s
    \begin{pmatrix}
      \phi \\
      \phi''
    \end{pmatrix}
    =
    D \mathcal{W}(0)
    \begin{pmatrix}
      \phi \\
      \phi''
    \end{pmatrix}
    + o(|(\phi, \phi'')|) \text{ as } |(\phi, \phi'')| \rightarrow 0.
  \end{equation}
  In what follows, we consider $(\phi', \phi^{(3)})$ when $(\phi, \phi'') \in \p \cC \setminus \{0\}$ and $|(\phi, \phi'')|$ is sufficiently small.

  If $\phi = 0$ and $\phi'' < 0$, then $\phi' < 0$.
  Now, if $\phi'' = 2\sqrt{6} \sin\phi = U_0(\phi)$, where $\phi > 0$, we have
  \begin{equation*}
    \phi^{(3)} - U_0'(\phi) \phi' > 0.
  \end{equation*}
  We combine these observations with the invariance of~\eqref{eq:5dODE} under $\phi \mapsto -\phi$ to observe for a sufficiently small neighbourhood of $W^u(0)$ around the origin that $(\phi, \phi'') \in \cC \cup -\cC$ is a negative invariant property, and if $(\phi, \phi'') \in \p (\cC \cup -\cC) \setminus \{0\}$ at some $s$-time then for all earlier $s$-times we have $(\phi, \phi'') \in \cC \cup -\cC$.

  Now we move onto proving the remainder of this lemma.
  For this we have a $s_0 < s_f$ such that $(\phi(s), \phi''(s)) \in \R^2 \setminus \overline{(\cC \cup -\cC)}$ for all $s < s_0$.
  Note that $\phi(s) \neq 0$ for all $s < s_0$.

  Due to the invariance of~\eqref{eq:5dODE} under $\phi \mapsto -\phi$, we assume without loss of generality that $\phi(s) > 0$ for all $s < s_0$.
  We can find a $s_1 < s_0$ so that $|(\phi(s_1), \phi''(s_1))|$ is as small as we wish, and $\phi''(s_1) \geq U_0(\phi(s_1)) \geq 3 \phi(s_1)$.
  From~\eqref{eq:ReducedODE}, we have $\phi'(s_1) > 0$ and
  \begin{equation*}
    \phi^{(3)}(s_1) - 3 \phi'(s_1) > 0.
  \end{equation*}
  An application of \Cref{lem:ConeBlowup} finishes the proof.
\end{proof}

Now the following theorem easily follows from combining \Cref{thm:BiharmBlowup}, \Cref{cor:ExitMiddleBlowup}, \Cref{lem:OrbitsInLocalUnstableManifold}, and the invariance of~\eqref{eq:5dODE} under $\phi \mapsto -\phi$.

\begin{theorem}\label{thm:BlowupOrExistsForever}
  Let $\phi \in C^\infty((-\infty, s_f); \R)$ a non-trivial solution of~\eqref{eq:5dODE} in $W^u(0)$, where $s_f \in (-\infty, \infty]$ is the maximal time of existence of $\phi$.
  Then one of the following statements is true:
  \begin{enumerate}[(i)]
  \item
    $s_f = \infty$ and $(\phi(s), \phi''(s)) \in \cC \cup -\cC$ for all $s \in \R$; or
  \item
    there exists a $s_\star \in (-\infty, s_f)$ such that $\phi''(s_\star) = 2 \sqrt{6}$ (or, $\phi''(s_\star) = -2 \sqrt{6}$), and hence $s_f < \infty$ and $\phi^{(i)}(s) \rightarrow \infty$ (resp., $\phi^{(i)}(s) \rightarrow -\infty$) as $s \nearrow s_f$ for $i \in \{0, 1, 2, 3\}$.
  \end{enumerate}
\end{theorem}

Next, we show the existence of a heteroclinic orbit, contained within $\cC$, connecting the origin and $(\pi/2, 0, 0, 0)$.

\begin{theorem}\label{thm:HeteroclinicExists}
  There exists $\phi \in C^\infty(\R; \R)$ which is a solution of~\eqref{eq:5dODE} in $W^u(0)$ such that $(\phi(s), \phi''(s)) \in \cC$ for all $s \in \R$, and $(\phi(s), \phi'(s), \phi''(s), \phi^{(3)}(s)) \rightarrow (\pi/2, 0, 0, 0)$ as $s \rightarrow \infty$.
\end{theorem}

\begin{proof}
  By \Cref{lem:LimitWhenBoundedForAllTime} it suffices to show that there exists a $\phi \in C^\infty(\R; \R)$ solving~\eqref{eq:5dODE} in $W^u(0)$ such that $(\phi(s), \phi''(s)) \in \cC$ for all $s \in \R$.
  We will proceed via contradiction, and hence we assume that there is no such $\phi$.

  The proof of \Cref{lem:OrbitsInLocalUnstableManifold} describes the dynamics of~\eqref{eq:5dODE} in $W^u(0)$ locally around zero.
  Therefore, as $s \rightarrow -\infty$ and $(\phi(s), \phi'(s), \phi''(s), \phi^{(3)}(s))$ enters a sufficiently small neighbourhood of $W^u(0)$ around the origin,~\eqref{eq:W0LocalGraphParam} gives $(\phi'(s), \phi^{(3)}(s)) = \mathcal{W}((\phi(s), \phi''(s)))$.

  We define $\iota : \R \rightarrow W^u(0)$ via
  \begin{equation*}
    \iota(\theta) := (\eps_0 \cos\theta, \mathcal{W}_1((\eps_0 \cos\theta, \eps_0 \sin\theta)), \eps_0 \sin\theta, \mathcal{W}_2((\eps_0 \cos\theta, \eps_0 \sin\theta))),
  \end{equation*}
  where $\eps_0 \in (0, 0.1)$ is sufficiently small so that $\iota$ is well-defined.
  Let $\theta_0$ be the unique value in $[0, \pi/2]$ such that $(\eps_0 \cos\theta_0, \eps_0 \sin\theta_0) \in \p \cC$.

  We consider $\phi_\theta \in C^\infty((-\infty, s_f); \R)$, for $\theta \in [-\pi/2, \theta_0]$, where $\phi_\theta$ is the solution to~\eqref{eq:5dODE} with $(\phi_\theta(0), \phi_\theta'(0), \phi_\theta''(0), \phi_\theta^{(3)}(0)) = \iota(\theta)$, and $s_f \in (0, \infty)$ is the maximal time of existence of $\phi_\theta$.
  From the proof of \Cref{lem:OrbitsInLocalUnstableManifold}, we know for $\theta \in [-\pi/2, \theta_0]$ that $(\phi_\theta(s), \phi_\theta''(s)) \in \cC$ for $s < 0$, as long as $\eps_0$ is sufficiently small, which we ensure.

  For $\theta \in [-\pi/2, \theta_0]$, we let $\tau(\theta)$ be the first $s \in [0, s_f)$ such that $|\phi_\theta''(s)| \geq 2\sqrt{6}$, if there is no such $s$ then we set $\tau(\theta) = \infty$.

  Note that for $\phi$, a solution to~\eqref{eq:5dODE}, it cannot blowup while $|\phi''|$ is bounded.
  From this and \Cref{cor:ExitMiddleBlowup}, we know that $\tau(\theta) < \infty$ for all $\theta \in [-\pi/2, \theta_0]$.
  We define $g : [-\pi/2, \theta_0] \rightarrow \{-1, 1\}$ via $g(\theta) := \sgn(\phi_\theta''(\tau(\theta)))$.
  For $\theta \in [-\pi/2, \theta_0]$, \Cref{thm:BiharmBlowup} implies that $|\phi_\theta''(s)| \geq 2\sqrt{6}$ for $s \in [\tau(\theta), s_f)$.
  Therefore, from continuous dependence on initial conditions we have that $g$ is continuous, and hence $g$ must be constant.
  However, \Cref{lem:LeftSideExit} implies $g(\theta_0) = 1$, and \Cref{lem:RightSideExit} implies $g(-\pi/2) = -1$, and hence we have our desired contradiction.
\end{proof}

Finally, \Cref{thm:EntireSolutionsAreHeteroclinic} is a simple consequence of combining \Cref{lem:LimitWhenBoundedForAllTime}, \Cref{thm:BlowupOrExistsForever}, and \Cref{thm:HeteroclinicExists}.

\section{Solutions \texorpdfstring{to~\eqref{eq:PsiProblem}}{} yield smooth biharmonic maps}\label{sect:GivesSmoothBiharmonicMaps}
In this section we prove \Cref{thm:SmoothBiharmonicMap}, that is, we show that solutions to~\eqref{eq:PsiProblem} give rise to smooth biharmonic maps.
Our arguments follow~\cite[Lemma 13 and Lemma 15]{MKC15} very closely.
Using this result we finally prove \Cref{thm:InfiniteWrapAround}.

The next lemma obtains estimates on the derivatives of $u = \Upsilon(\psi)$, where $\psi$ is a solution to~\eqref{eq:PsiProblem}.
A consequence of this is that such $u$ are indeed in the energy space $H^2(B^d(0, 1); S^d)$, where $d \in \{5, 6, 7\}$.

\begin{lemma}\label{lem:uSobolev}
  Let $d \in \{5, 6, 7\}$, $\psi \in C([0,1];\R) \cap C^{\infty}((0,1];\R)$, with $\psi(0) = 0$, be a solution to~\eqref{eq:PsiProblem}, and
  \begin{equation*}
    u = \Upsilon(\psi) \in C(\overline{B^d(0,1)}; S^d) \cap C^{\infty}(\overline{B^d(0,1)} \setminus \{0\};S^d).
  \end{equation*}
  Then $u \in W^{2, p}(B^d(0, 1); S^d)$ for all $p \in (1, \infty)$.
\end{lemma}

\begin{proof}
  We fix an $\eps \in (0, 1/3)$, and set $\phi(s) = \psi(e^s)$.
  In this proof we allow all constants to implicitly depend on $\eps$ and $\psi$.

  Recall that $\phi$ solves~\eqref{eq:FullODE} and is in $W^u(0)$ by \Cref{thm:BiharmonicInUnstable}.
  Using~\eqref{eq:W0LocalParam} and~\eqref{eq:W0ReducedODE} we can find a $s_0 < 0$, depending on $\eps$ and $\psi$, such that we can write
  \begin{equation*}
    (\phi(s), \phi'(s), \phi''(s), \phi^{(3)}(s)) = \Gamma(z(s)),
  \end{equation*}
  where
  \begin{equation*}
    \p_s z =
    \begin{pmatrix}
      1 & 0 \\
      0 & 3
    \end{pmatrix}
    z
    + \cO(|z|^3) \text{ as } z \rightarrow 0.
  \end{equation*}
  From this we have
  \begin{equation}\label{eq:zDecayEstimates}
    |z(s)| \leq C e^{(1-\eps) s} \text{ and } |z_2(s)| \leq C e^{3(1-\eps) s} \text{ for } s < s_0.
  \end{equation}
  For $r \in (0, 1]$, we have
  \begin{equation*}
    \p_r^2 \psi(r)
    =
    \frac{\phi''(\log r) - \phi'(\log r)}{r^2}.
  \end{equation*}
  Using~\eqref{eq:W0LocalParam} and~\eqref{eq:zDecayEstimates}, we have, for $s < s_0$,
  \begin{equation*}
    |\p_r^2 \psi(e^{s})|
    =
    e^{-2 s} |\phi''(s) - \phi'(s)|
    \leq
    C e^{(1 - 3\eps) s}.
  \end{equation*}
  From this we have that $\psi \in C^2([0, 1]; \R)$, and, for $r \in [0, 1]$,
  \begin{equation}\label{eq:psiGrowthEstimates}
    \psi''(r) = \cO(r^{1 - 3\eps}),
    \psi'(r) = \psi'(0) + \cO(r^{2 - 3\eps}), \text{ and }
    \psi(r) = \psi'(0) r + \cO(r^{3 - 3\eps}).
  \end{equation}
  For $x \neq 0$, we have
  \begin{equation*}
    \Delta u(x) = \left(\frac{x}{|x|} L_0 f_0(r), L_1 f_1(r) \right),
  \end{equation*}
  where $f_0(r) = \sin \psi(r)$, $f_1(r) = \cos \psi(r)$,
  \begin{align*}
    L_0 f &:= f'' + \frac{d - 1}{r} f' - \frac{d - 1}{r^2} f, \text{ and}
    \\
    L_1 f &:= f'' + \frac{d - 1}{r} f'.
  \end{align*}
  Using~\eqref{eq:psiGrowthEstimates}, for $r \in (0, 1]$, we have $|L_0 f_0(r)| \leq C$ and $|L_1 f_1(r)| \leq C$, and hence $|\Delta u(x)| \leq C$ for $x \neq 0$.
  Therefore, $\Delta u \in L^p(B^d(0, 1); \R^{d+1})$ for all $p \in (1, \infty)$.
  Standard elliptic theory gives $u \in W^{2, p}(B^d(0, 1); S^d)$ for all $p \in (1, \infty)$.
\end{proof}

We can now prove \Cref{thm:SmoothBiharmonicMap}.
See~\cite{GastelZorn12} for a similar approach in a slightly different situation.

\begin{proof}[Proof of \Cref{thm:SmoothBiharmonicMap}]
  First, we show that $u$ is weakly biharmonic.
  We let $\eta \in C^{\infty}_c(B^d(0,1); \R^{d + 1})$ be arbitrary.
  We wish to show that
  \begin{equation*}
    \p_t|_{t = 0} E_2(\Pi(u + t \eta)) = 0,
  \end{equation*}
  where we define $\Pi : \R^{d + 1} \rightarrow S^d$ via $\Pi(x) := x/|x|$.

  From~\cite[(2.1) and (2.2)]{Strzelecki03}, we have
  \begin{equation*}
    \p_t|_{t = 0} E_2(\Pi(u + t \eta))
    =
    2 \int_{B^d(0,1)} \left( \Delta u \cdot \Delta \eta - \sum_{\gamma=1}^{d + 1} \Delta u^{\gamma} \Delta \left( u^{\gamma} u \cdot \eta\right) \right) \, dx.
  \end{equation*}
  We let $\omega \in C^{\infty}(\R^d; [0,1])$ be such that $\omega(x) = 1$ for $|x| \leq 1/2$ and $\omega(x) = 0$ for $|x| \geq 3/4$.
  For $R > 0$, we set $\omega_R(x) = \omega(x/R)$.
  We have
  \begin{equation*}
    \p_t|_{t = 0} E_2(\Pi(u + t \eta))
    =
    \p_t|_{t = 0} E_2(\Pi(u + t (\omega_R \eta)))
    +
    \p_t|_{t = 0} E_2(\Pi(u + t ((1-\omega_R)\eta))).
  \end{equation*}
  \Cref{lem:uSobolev}, gives us
  \begin{equation*}
    \p_t|_{t = 0} E_2(\Pi(u + t (\omega_R \eta)))
    =
    o(1) \text{ as } R \searrow 0.
  \end{equation*}
  Next, we turn our attention towards $\p_t|_{t = 0} E_2(\Pi(u + t ((1-\omega_R)\eta)))$.
  Since the support of $(1-\omega_R)\eta$ is bounded away from the origin, and $u$ is smooth and satisfies the Euler-Lagrange equation~\eqref{eq:BiEulerLagrange} away from the origin, we have $\p_t|_{t = 0} E_2(\Pi(u + t ((1-\omega_R)\eta))) = 0$.
  Therefore, $\p_t|_{t = 0} E_2(\Pi(u + t \eta)) = o(1)$ as $R \searrow 0$ which gives the desired result after taking the limit $R \searrow 0$.

  Since $u \in C(\overline{B^d(0, 1)}; S^d)$, higher interior regularity for weakly-biharmonic maps, for example, see~\cite[Theorem 5.1]{CWY99}, gives $u \in C^\infty(\overline{B^d(0, 1)}; S^d)$.
\end{proof}

Finally, we prove \Cref{thm:InfiniteWrapAround}.

\begin{proof}[Proof of \Cref{thm:InfiniteWrapAround}]
  From~\eqref{eq:W0LocalGraphParam} we know that $W^u(0)$, locally around the origin, is a graph over the $(\phi, \phi'')$-plane.
  Therefore, there exists a $\phi \in C^\infty((-\infty, s_f); \R)$ which solves~\eqref{eq:5dODE} in $W^u(0)$, where $s_f \in (-\infty, \infty]$ is the maximal time of existence of $\phi$, and a $s_0 \in (-\infty, s_f)$ such that $(\phi(s_0), \phi''(s_0)) \not\in \cC \cup -\cC$.
  \Cref{thm:BlowupOrExistsForever} shows that $s_f < \infty$, and that we can assume $\phi(s) \rightarrow \infty$ as $s \nearrow s_f$.
  Since~\eqref{eq:5dODE} is autonomous, we may assume that $s_f = 0$.

  We can define a $\psi \in C([0, 1); \R) \cap C^\infty((0, 1); \R)$ via $\psi(r) := \phi(e^r)$ which solves~\eqref{eq:PsiProblem}.
  Finally, the dilation invariance of~\eqref{eq:PsiProblem} combined with \Cref{thm:SmoothBiharmonicMap} shows that $u = \Upsilon(\psi) \in C^\infty(B^5(0, 1); S^5)$ is biharmonic.
\end{proof}

\bibliographystyle{plain}
\bibliography{master}
\end{document}